\documentclass[11pt]{amsart}
\usepackage{amsmath,mathrsfs, amssymb, amsthm}
\usepackage[all]{xy}
\usepackage{hyperref} 
\hypersetup{ 
colorlinks=true, 
linkbordercolor={1 1 1}, 
citebordercolor={1 1 1} 
} 

\newcommand{\linenumberscmd}{}

\usepackage{svn}
\SVN $Date: 2010-09-30 10:52:03 -0400 (Thu, 30 Sep 2010) $
\SVN $Revision: 37 $

\newtheorem{thm}{Theorem}[section]
\newtheorem{cor}[thm]{Corollary}
\newtheorem{lem}[thm]{Lemma}
\newtheorem{prop}[thm]{Proposition}
\newtheorem{proposition}[thm]{Proposition}
\newtheorem{claim}[thm]{Claim}
\theoremstyle{definition}
\newtheorem{defn}[thm]{Definition}
\newtheorem{quest}[thm]{Question}

\newtheorem{rem}[thm]{Remark} 
\numberwithin{equation}{thm}

\newcommand{\N}{\mathbb{N}}
\newcommand{\setN}{\mathbb{N}}

\newcommand{\res}{\mathbin{\upharpoonright}}

\newcommand{\convergesto}{\mathbin{\downarrow =}}

\newcommand{\M}{\mathscr{M}}
\newcommand{\Lang}{\mathsf{L}}

\newcommand{\0}{{\bf 0}}

\newcommand{\om}{\omega}

\newcommand{\emp}{\emptyset}

\newcommand{\RCA}{\mathsf{RCA}_0}
\newcommand{\RCAo}{\mathsf{RCA}_0}
\newcommand{\ACA}{\mathsf{ACA}_0}
\newcommand{\ACAo}{\mathsf{ACA}_0}
\newcommand{\WKL}{\mathsf{WKL}_0}
\newcommand{\WKLo}{\mathsf{WKL}_0}
\newcommand{\ATRo}{\mathsf{ATR}_0}

\newcommand{\CA}{\mathsf{CA}_0}

\newcommand{\QF}{\mathsf{QF}}

\newcommand{\RT}{\mathsf{RT}}
\newcommand{\SRT}{\mathsf{SRT}}
\newcommand{\CAC}{\mathsf{CAC}}
\newcommand{\ADS}{\mathsf{ADS}}
\newcommand{\SADS}{\mathsf{SADS}}
\newcommand{\DNR}{\mathsf{DNR}}
\newcommand{\COH}{\mathsf{COH}}

\newcommand{\ZL}[1]{\mathsf{ZL}\text{-}\mathsf{#1}}
\newcommand{\IP}{\mathsf{IP}}
\newcommand{\FCP}{\mathsf{FCP}}
\newcommand{\AMT}{\mathsf{AMT}}
\newcommand{\OPT}{\mathsf{OPT}}
\newcommand{\NCE}{\mathsf{NCE}}
\newcommand{\cl}{\operatorname{cl}}

\let\oldsetminus-
\renewcommand{\setminus}{{\oldsetminus}}

\newcommand{\chat}[1]{\widehat{#1}}

\title{Reverse mathematics and equivalents of the axiom of choice}
\author{Damir D. Dzhafarov}
\address{\hspace*{-\parindent}Department of Mathematics\\
  University of Chicago\\
  5734 South University Avenue\\
  Chicago, Illinois 60637 USA} \email{damir@math.uchicago.edu}
\author{Carl Mummert}
\address{\hspace*{-\parindent}Department of Mathematics\\
  Marshall University\\
  1 John Marshall Drive\\
  Huntington, West Virginia 25725 USA} \email{mummertc@marshall.edu}
\thanks{The authors are grateful to Denis Hirschfeldt, Antonio
  Montalb\'{a}n, and Robert Soare for valuable comments and
  suggestions.  The first author was partially supported by an NSF
  Graduate Research Fellowship.}

\date{September 30, 2010} 

\begin{document}

\begin{abstract}
  We study the reverse mathematics of countable analogues of several
  maximality principles that are equivalent to the axiom of choice in
  set theory.  Among these are the principle asserting that every
  family of sets has a $\subseteq$-maximal subfamily with the finite
  intersection property and the principle asserting that if $\varphi$ is a
  property of finite character then every set has a
  $\subseteq$-maximal subset of which $\varphi$ holds. We show that these
  principles and their variations have a wide range of strengths in
  the context of second-order arithmetic, from being equivalent to
  $\mathsf{Z}_2$ to being weaker than $\mathsf{ACA}_0$ and
  incomparable with $\mathsf{WKL}_0$.  In particular, we identify a
  choice principle that, modulo $\Sigma^0_2$ induction, lies strictly
  below the atomic model theorem principle $\mathsf{AMT}$
 and implies the omitting partial
  types principle~$\mathsf{OPT}$.  
\end{abstract}

\maketitle

\tableofcontents

\linenumberscmd

\section{Introduction} 

A large number of statements in set theory are equivalent to the axiom
of choice over Zermelo--Fraenkel set theory ($\mathsf{ZF}$). In this
paper, we examine what happens when some of these statements are
interpreted in the setting of second-order arithmetic, where the only
``sets'' available are sets of natural numbers. This interpretation
allows us to study computability-theoretic and proof-theoretic aspects
of choice principles in the spirit of reverse mathematics. Our results
show that the re-interpreted statements need not be trivial, as might
be suspected.  Instead, these principles demonstrate a
wide range of reverse mathematical strengths.

The history of the axiom of choice is presented in detail by
Moore~\cite{Moore-1982}. The main facet of interest for our
purposes is that,
 after Zermelo introduced the axiom of choice in 1904, set
theorists began to obtain results proving other set-theoretic
principles equivalent to it (relative to choice-free axiomatizations
of set theory). These equivalence results, and their further
development, now constitute a program in set theory, which
has been documented in
detail by Jech~\cite{Jech-1973} and by Rubin and Rubin~\cite{RR-1970,
  RR-1985}.

This program provides us with a large collection of statements from
which to choose.  We begin in Section~\ref{S:ZL} with Zorn's lemma,
which is perhaps the most well-known equivalent of the axiom of choice
but which turns out to be of only limited interest in second-order
arithmetic. In Sections~\ref{S:FIP} and~\ref{S:FCP}, we turn to
other maximality principles with more complex and interesting
behavior. Our focus is on statements closely related to the following two
equivalents of the axiom of choice:
\begin{itemize}
\item every family of sets has a $\subseteq$-maximal subfamily with
  the finite intersection property;
\item if $\varphi$ is a property of finite character and $A$ is any set,
  there is a $\subseteq$-maximal subset $B$ of $A$ such that $B$ has property~$\varphi$.
\end{itemize}

We avoid studying principles that concern countable
well-orderings. Such principles have been thoroughly explored in the
context of reverse mathematics by Friedman and Hirst~\cite{HF-1990}
and by Hirst~\cite{Hirst-2005}.  We also do not study direct
formalizations of choice principles in arithmetic. These have been
studied by Simpson~\cite[Section VII.6]{Simpson-2009}. 

The rest of this section is devoted to a brief overview of
second-order arithmetic and reverse mathematics.  We refer the reader
to \mbox{Simpson~\cite{Simpson-2009}} for complete details
on second-order arithmetic and to Soare~\cite{Soare-1987}
for background information on computability theory.

\subsection{Second-order arithmetic}

Second-order arithmetic is, intuitively, a weak form of type theory in
which there are only two kinds of primitive objects: natural numbers
and sets of natural numbers. This system is sufficiently expressive
that many theorems of classical mathematics can be formalized within
it, provided that the theorems are put in an arithmetical context
through appropriate coding conventions and countability assumptions.

We work in the language $\Lang_2$ of second-order arithmetic, which
has the signature $\langle 0, 1, +, \times, <, =_{\setN}, \in\rangle$.
Equality for sets of numbers is defined by extensionality: $X = Y$ is
an abbreviation for $\forall n\, (n \in X \leftrightarrow n \in Y)$.
The set of $\Lang_2$ formulas is ramified into the arithmetical and
analytical hierarchies, which are used to define induction and
comprehension schemes.

The full second-order \textit{induction scheme} consists of every
instance of
\[ 
(\varphi(0) \land (\forall n)[\varphi(n) \to \varphi(n+1)]) 
  \to (\forall n)\, \varphi(n),
\]
in which $\varphi$ is an $\Lang_2$-formula, possibly with set
parameters.  If $\Gamma$ is $\Sigma^i_n$ or $\Pi^i_n$ for some $i \in
\{0,1\}$ and $n \geq 0$, the scheme of \emph{$\Gamma$ induction}
($\mathsf{I}\Gamma$) consists of the restriction of the induction
scheme to formulas in~$\Gamma$.

The full second-order \textit{comprehension scheme} consists of every
instance of
\[ 
  (\exists X)(\forall n)[n \in X \leftrightarrow \varphi(n)]
\] 
in which $\varphi$ is an $\Lang_2$-formula that does not
mention~$X$ but may have other set parameters.  If $\Gamma$ is
$\Sigma^i_n$ or $\Pi^i_n$, where $i \in \{0,1\}$ and $n \geq 0$, the
scheme of \textit{$\Gamma$ comprehension}
($\Gamma\text{-}\mathsf{CA}$) consists of the restriction of the
induction scheme to formulas in~$\Gamma$.  We also have the scheme of
\emph{$\Delta^i_n$ comprehension} ($\Delta^i_n\text{-}\mathsf{CA}$),
which contains every instance of
\[ 
  (\forall n)[\varphi(n) \leftrightarrow \psi(n)] 
   \to (\exists X)(\forall n)[n \in X \leftrightarrow \phi(n)]
\] 
in which $\varphi$ is $\Sigma^i_n$, $\psi$ is $\Pi^i_n$, and
neither of these formulas mentions~$X$.

The theory $\mathsf{Z}_2$ of (full) \textit{second-order arithmetic} includes
the axioms of a discrete ordered ring, the full comprehension scheme,
and the full induction scheme.

Semantic interpretations of $\Lang_2$-theories are given by
\textit{$\Lang_2$-structures}.  A general $\Lang_2$-structure $\M$
includes a set $\setN^\M$ of ``numbers'', a collection
$\mathcal{S}^\M$ of ``sets'', and interpretations of the symbols of
$\Lang_2$ using $\setN^\M$ and $\mathcal{S}^\M$. An
$\Lang_2$-structure $\M$ is an \textit{$\omega$-model} if $\setN^\M$
is the set $\omega = \{0,1,2,\ldots\}$ of standard natural numbers,
$\mathcal{S}^{\M} \subseteq \mathcal{P}(\omega)$, and all the symbols
of $\Lang_2$ are given their standard interpretations.  We identify an
$\omega$-model with the collection of subsets of $\omega$ that it
contains.  As usual, the notation $\M \models \varphi$ indicates that
the formula $\varphi$ (which may have parameters from $\M$) is true
in~$\M$.

\subsection{Subsystems}

Fragments of $\mathsf{Z}_2$ are called \emph{subsystems of
  second-order arithmetic}.  The program of reverse mathematics seeks
to characterize statements in the language of second-order arithmetic
according to the weakest subsystems that can prove them. These
characterizations are obtained by proving a statement within a certain
subsystem, and then proving in a weak base system that the statement
implies all the axioms of that subsystem.

As is common in reverse mathematics, we will use the subsystem $\RCAo$ for this weak
base system.   $\RCAo$ includes the axioms of a 
discrete ordered semiring, $\Sigma^0_1$ induction, 
and $\Delta^0_1$ comprehension.  Intuitively, this subsystem corresponds to 
computable mathematics, and in fact it is satisfied by the 
$\omega$-model $\mathsf{REC}$ containing only computable sets.  In this 
sense, $\RCAo$ is very weak. Nevertheless, it is able to establish many 
elementary properties of the natural numbers. 

Two countable forms of equivalents of the axiom of choice are already
provable in $\RCAo$.  These are the principle that every set of
natural numbers can be well ordered and the
principle that every sequence of nonempty sets of natural numbers has
a choice function.  We will show that several other
equivalents of the axioms of choice
require stronger subsystems to prove.

These stronger systems are obtained by adding stronger set-existence
axioms to $\RCAo$. The main ones we will be interested in are the
following:
\begin{itemize}
\item $\ACAo$ is the subsystem obtained by adding the comprehension
scheme for arithmetical formulas;
\end{itemize} 
and for each $n \geq 1$,
\begin{itemize}
\item $\Pi^1_n\text{-}\CA$ is the subsystem obtained by adding the
scheme of $\Pi^1_n$ comprehension;
\item $\Delta^1_n\text{-}\CA$ is the subsystem obtained by adding the
scheme of $\Delta^1_n$ comprehension.
\end{itemize} 
There are two important subsystems that do not directly
correspond to restrictions of the second-order comprehension scheme.
The first of these, $\WKLo$, consists of $\RCAo$ along
with a single axiom, known as \textit{weak K\"{o}nig's lemma}, which
states any infinite subtree of $2^{<\setN}$ contains an infinite
path. The second, $\ATRo$, consists of $\RCAo$ along with an axiom scheme
that states that any arithmetically-defined functional $F\colon
2^{\setN} \to 2^{\setN}$ may be iterated along any countable
well-ordering, starting with any set. We will not make use of 
$\ATRo$ in this paper. 

The following theorem summarizes the well-known relations between
the subsystems we have mentioned, in terms of provability.
For subsystems $T$ and $T'$, we write $T < T'$ if every axiom
of $T$ is provable in $T'$ but some axiom of $T'$ is not provable
in~$T$.

\begin{thm}
We have
\[ 
\RCAo < \WKLo < \ACAo < \Delta^1_1\text{-}\CA < \ATRo <\Pi^1_1\text{-}\CA,
\] 
and for each $n \geq 1$,
\[ 
\Pi^1_n\text{-}\CA < \Delta^1_{n+1}\text{-}\CA < \Pi^1_{n+1}\text{-}\CA.
\]
\end{thm}

\section{Zorn's lemma}\label{S:ZL}

Zorn's lemma is one of the best known equivalents of the axiom of
choice, so we begin by studying the strength of countable versions of
this principle. The reverse mathematics results in this section are
relatively elementary, providing a warm-up for the more technical
results of the following sections.

Working in $\RCAo$, we define a \textit{countable poset} to be a set
$P \subseteq \setN$ with a reflexive, antisymmetric, transitive
relation~$\leq_P$. As usual, we may freely convert $\leq_P$ into an
irreflexive, transitive relation~$<_P$.

\begin{defn} 
The following principles are defined in $\RCA$.

\begin{list}{\labelitemi}{\leftmargin=0em}\itemsep2pt
\item[]($\ZL{1})$ If a nonempty countable poset has the property that
every linearly ordered subset is bounded above, then every element of
the poset is below some maximal element.  \medskip

\item[]($\ZL{2}$) If a nonempty countable poset has the property that
every linearly ordered subset is bounded above, then there is a
nonempty set consisting of the maximal elements of the poset.
\medskip

\item[]($\ZL{3}$) If a nonempty countable poset has the property that
every linearly ordered subset is bounded above, then there is a
function that assigns to each element of the poset a maximal element
above it.
\end{list}
\end{defn}

Of these three principles, $\ZL{1}$ is the most natural countable
analogue of Zorn's lemma, but we will see that it is already provable
in $\RCAo$. We will show that $\ZL{2}$ is equivalent to $\ACAo$ over
$\RCAo$, as might be expected. Principle $\ZL{3}$ is of greater
interest; it can be viewed as a uniform version of $\ZL{1}$. We will
show it is also equivalent to $\ACAo$ over $\RCAo$.

\begin{thm}\label{zlone} 
$\ZL{1}$ is provable in $\RCA$.
\end{thm}

\begin{proof} 
Working in $\RCAo$, let $\langle P,\leq_P \rangle$ be a
countable poset in which every linearly ordered subset of $P$ is
bounded above.  Write $P = \langle p_i : i \in \setN \rangle$. We will
build a sequence $\langle q_i : i \in \setN\rangle$ by induction.  Let
$q_0$ be an arbitrary element of~$P$.  At stage $i+1$, if $q_{i} <_P
p_i$ then put $q_{i+1} = p_i$, and otherwise put $q_{i+1} = q_i$. This
inductive construction can be carried out in $\RCAo$.  Moreover, a
$\Pi^0_1$ induction in $\RCAo$ shows that if $i < j$ then $q_i \leq_P
q_j$.

Let $L = \{ q_i : i \in \N\}$. To decide if a fixed $p_i \in P$ is in
$L$, it is only necessary to simulate the construction up to stage
$i+1$. Therefore $L$ is a $\Delta^0_1$ set, and so $\RCAo$ proves that
$L$ exists.  Moreover, $L$ is linearly ordered; if two elements of $L$
are incomparable, then two elements of the original sequence $\langle
q_i : i \in \N\rangle$ are incomparable, which is impossible.

By assumption, there is some $i \in \setN$ such that $p_i$ is an upper
bound for~$L$. In particular, it must be that $q_{i} <_P p_i$, which
means by construction that $p_i = q_{i+1} \in L$.  Moreover, because
$p_i$ is an upper bound for $L$, it must be that $q_{i+j} = p_i$ for
all $j \geq 1$.

Now suppose there is some $p_j \in P$ with $p_i <_P p_j$. It cannot be
that $j < i$, because this would imply $p_j \leq_P q_i \leq_P
p_i$. However, if $i < j$ then, at stage $j$, the construction would
select $q_{j+1} = p_j$, contradicting our result that $q_{j+1} = p_i$.
Thus $p_i$ is a maximal element above~$q_0$.
\end{proof}

\begin{thm}
Each of $\ZL{2}$ and $\ZL{3}$ is equivalent to $\ACA$ over $\RCA$.
\end{thm}

\begin{proof} 
For any countable poset $P$ satisfying the hypothesis of
$\ZL{2}$, the set of maximal elements of $P$ is definable by an
arithmetical formula and is nonempty by Theorem~\ref{zlone}.  Thus,
$\ACAo$ implies $\ZL{2}$.

Next, we show that $\ZL{2}$ implies $\ZL{3}$ over $\RCA$.  Let
$\langle P, \leq_P \rangle$ be any countable poset such that every
element of $P$ is below at least one maximal element, and by $\ZL{2}$
let $M$ be the set of maximal elements of~$P$.  Define a function $m
\colon P \to P$ by the rule
\[
m(p) = q \Leftrightarrow (q \in M) \land (p \leq_P q) \land
(\forall r <_{\mathbb{N}} q) [ p \leq_P r \to r \not \in M].
\] 
Then $m$ is a function with domain $P$ such that for each $p$,
$m(p)$ is a maximal element with $p \leq_P m(p)$. Moreover, the
definition of $m$ is $\Delta^0_0$ relative to $M$ and $\leq_P$, so we
can form $m$ in $\RCA$.

Finally, we show that $\ZL{3}$ implies $\ACA$ over $\RCA$.  Fix any
one-to-one function~$f$. We will construct a poset $\langle P, \leq_P
\rangle$ as follows. Let $P = \{ p_{i,s} : i,s \in \setN\}$.  The
order $\leq_P$ on $P$ is defined by cases.  If $i \not = j$ then
$p_{i,s}$ and $p_{j,t}$ are incomparable for all $s,t \in
\setN$. Given $i,s,t \in \setN$, with $s \not = t$, define $p_{i,t}
<_P p_{i,s}$ to hold if either $f(s) = i$, or $f(t) \not = i$ and $t >
s$.  Thus, for a fixed $i$, if there is no $s$ with $f(s) = i$ then we
have a maximal chain
\[
 \cdots <_P p_{i,2} <_P p_{i,1} <_P p_{i,0} ,
\] 
while if $f(s) = i$ then, because $f$ is one-to-one, we have a
maximal chain
\[ 
\cdots <_P p_{i,2} <_P p_{i,1} <_P p_{i,0} <_P p_{i,s} .
\]
In particular, for each $i$, either $p_{i,0}$ is a maximal element
of $P$ or there is an $s$ with $f(s) = i$ and $p_{i,s}$ is a maximal
element of~$P$. (This gives, as a corollary, a direct reversal of
$\ZL2$ to $\ACA$ over $\RCA$.)

Now, working in $\RCA$, assume there is a function $m\colon P \to P$
taking each $p \in P$ to a $\leq_P$-maximal $q$ with $p \leq_P q$. Fix
$i \in \setN$. Either $m(p_{i,0}) = p_{i,0}$, in which case $i$ is in
the range of $f$ if and only if $f(0) = i$, or else $m(p_{i,0}) =
p_{i,s}$ for some $s > 0$, in which case $f(s) = i$.  Thus we have
\[ 
i \in \operatorname{range}(f) \Leftrightarrow (\exists s)[f(s) = i]
\Leftrightarrow (\forall s)[m(p_{i,0}) = p_{i,s} \Rightarrow f(s) =
i].
\] 
Therefore the range of $f$ exists by $\Delta^0_1$ comprehension.  This
completes the reversal.
\end{proof}

\section{Intersection properties}\label{S:FIP}

We next study several principles asserting that every countable
family of sets has a $\subseteq$-maximal subfamily with certain
intersection properties (see Definition~\ref{D:prop}).  We
will show that, although these principles are all equivalent to the
axiom of choice in set theory, they can have vastly different strengths
when formalized in second-order arithmetic. In particular, we find new
examples of principles weaker than $\ACAo$ and incomparable with $\WKLo$.

\begin{defn}\label{D:family}
  We define a \textit{family of sets} to be a sequence $A =
  \langle A_i : i \in \om \rangle$ of sets.  A family $A$ is {\em
    nontrivial} if $A_i \neq \emp$ for some $i \in \om$.  

  Given a family of sets $A$ and a set $X$, we say $A$ \emph{contains}
  $X$, and write $X \in A$, if $X = A_i$ for some $i \in \om$.  A
  family of sets $B$ is a {\em subfamily} of $A$ if every set
in $B$ is in~$A$, that is,  $(\forall
  i)(\exists j)[B_i = A_j]$.
  Two sets 
  $A_i, A_j \in A$ are \textit{distinct} if they differ 
  extensionally as sets.
\end{defn}

Our definition of a subfamily is intentionally weak; see 
Proposition~\ref{P:familydef} below and the remarks preceding it. 

\begin{defn}\label{D:prop} 
Let $A = \langle A_i : i \in \om \rangle$ be a family of sets and
fix $n \geq 2$.  Then $A$ has the
\begin{itemize}
\item {\em $D_n$ intersection property} if the intersection of any $n$ distinct sets in $A$ is empty.

\item {\em $\overline{D}_n$ intersection property} if the intersection of any $n$ distinct sets in $A$ is nonempty.

\item {\em F intersection property} if for every $m \geq 2$, the intersection of any $m$ distinct sets in $A$ is nonempty.

\end{itemize}
\end{defn}

\begin{defn}\label{D:maximal}
Let $A = \langle A_i : i \in \om \rangle$ and 
$B = \langle B_i : i \in \om \rangle$ be families of sets, and let $P$ be
any of the properties in Definition~\ref{D:prop}.  Then
$B$ is a {\em maximal} subfamily of $A$ with the $P$
intersection property if $B$ has the $P$ intersection property, and
for every subfamily $C$ of $A$ that does also, if $B$ is a subfamily
of $C$ then $C$ is a subfamily of $B$.
\end{defn}

It is straightforward to formalize
Definitions~\ref{D:family}--\ref{D:maximal} in $\RCA$.

Given a family $A = \langle A_i : i \in \om \rangle$ and some $J \in
\om^{\om}$, we use the notation $\langle A_{J(i)} : i \in \om \rangle$
for the subfamily $\langle B_i : i \in \om \rangle$ where $B_i =
A_{J(i)}$.  We call this the subfamily \emph{defined} by~$J$.  Given a
finite set $\{j_0,\ldots,j_n\} \subset \om$, we let $\langle
A_{j_0},\ldots,A_{j_n} \rangle$ denote the subfamily $\langle B_i : i
\in \N \rangle$ where $B_i = A_{j_i}$ for $i \leq n$ and $B_i =
A_{j_n}$ for $i > n$.  Note that such a subfamily can still contain $A_i$
for infinitely many $i$, because there could be a $j$ such that
$A_j = A_i$ for infinitely many $i$.  We call a subfamily of $A$
\emph{finite} if it contains only finitely many distinct $A_i$.

We are interested in the following maximality principles.

\begin{defn}\label{D:PIP} 
Let $P$ be any of the properties in Definition~\ref{D:prop}.  
The following principle is defined in $\RCA$.
 
\begin{list}{\labelitemi}{\leftmargin=0em}\itemsep2pt
\item[]($P\IP$) Every nontrivial family of sets has a
maximal subfamily with the $P$ intersection property.
\end{list}
\end{defn}

\noindent For $P = D_n$ and $P = \overline{D}_n$,
the set-theoretic principle corresponding to $P\IP$ is, in
the notation of Rubin and Rubin~\cite{RR-1985}, $\mathsf{M}\, 8 \,
(P)$.  For $P = F$, it is $\mathsf{M} \,14$.  For additional
references concerning the set-theoretic forms, and for proofs of their
equivalences with the axiom of choice, see Rubin and
Rubin~\cite[pp.~54--56,~60]{RR-1985}.

\begin{rem}\label{R:FIP_subfamilies} 
Although we do not make it an explicit part of the definition, all of 
the families $\langle A_i : i \in \om \rangle$ we construct in our results 
will have the property that for each $i$, $A_i$ contains $2i$ and otherwise 
contains only odd numbers.  This will have the advantage that if we are 
given an arbitrary subfamily $B = \langle B_i : i \in \om \rangle$ of some 
such family, we can, for each $i$, uniformly $B$-computably find a $j$ such
that $B_i = A_j$.  If $A$ is computable, each subfamily $B$ will then
be of the form $\langle A_{J(i)} : i \in \om \rangle$ for some $J \in
\om^\om$ with $J \equiv_T B$.
\end{rem}

\subsection{\texorpdfstring{Implications over $\RCA$, and equivalences to $\ACA$}{Implications over RCA0 and equivalents to ACA0}}\label{S:FIP_implications} 
The next sequence of propositions establishes the basic relations that
hold among the principles we have defined.  We begin with the
following upper bound on their strength.

\begin{prop}\label{prop_fip_first} 
For any property $P$ in Definition~\ref{D:prop}, $P\IP$ is provable in $\ACA$.
\end{prop}

\begin{proof} 
Suppose $A = \langle A_i : i \in \N \rangle$ is a
nontrivial family of sets.  If $A$ has a finite maximal subfamily with
the $P$ intersection property, then we are done.  Otherwise, we define
a function $p\colon \N \to \N$ as follows.  Let $p(0)$ be the least
$j$ such that $\langle A_j \rangle$ has the $P$ intersection property,
and given $i \in \N$, let $p(i+1)$ be the least $j > p(i)$ such that
$\langle A_{p(0)},\ldots,A_{p(i)}, A_j \rangle$ has the $P$
intersection property.  Then $p$ exists by arithmetical comprehension,
and by assumption it is total.  It is not difficult to see that $B =
\langle A_{p(i)} : i \in \N \rangle$ is a maximal subfamily of $A$
with the $P$ intersection property.
\end{proof}

\begin{prop}\label{prop_fip_second} 
For each standard $n \geq 2$, the following are provable in $\RCA$:
\begin{enumerate}
\item $F\IP$ implies $\overline{D}_n\IP$;
\item $\overline{D}_{n+1}\IP$ implies $\overline{D}_n\IP$.
\end{enumerate}
\end{prop}

\begin{proof} 
To prove (1), let $A = \langle A_i : i \in \N \rangle$
be a nontrivial family of sets.  We may assume that $A$ has no finite
maximal subfamily with the $\overline{D}_n$ intersection property.
Define a new family $\widetilde{A} = \langle \widetilde{A}_i : i \in
\N \rangle$ by recursion as follows.  For all $i \neq j$, let $2i \in
\widetilde{A}_i$ and $2j \notin \widetilde{A}_i$.  Now suppose $s$ is
such that the $\widetilde{A}_i$ have been defined precisely on the odd
numbers less than $2s+1$.  Consider all finite sets $F \subseteq
\{0,\ldots,s\}$ such that $|F| \geq n+1$ and for every $F' \subseteq
F$ of size $n$ there is an $x \leq s$ belonging to $\bigcap_{i \in F'}
A_i$.  If no such $F$ exists, enumerate $2s+1$ into the complement of
$\widetilde{A}_i$ for all~$i$.  Otherwise, list these sets as
$F_0,\ldots,F_k$.  For each $j \leq k$, enumerate $ 2(s+j)+1$ into
$\widetilde{A}_i$ if $i \in F_j$, and into the complement of
$\widetilde{A}_i$ if $i \notin F_j$.

The family $\widetilde{A}$ exists by $\Delta^0_1$ comprehension, and 
is nontrivial by construction.  Let $\widetilde{B} = \langle
\widetilde{B}_i : i \in \N \rangle$ be a maximal subfamily of
$\widetilde{A}$ with the $F$ intersection property.  Now
each $\widetilde{B}_i$ contains exactly one even number, and if $2j \in
\widetilde{B}_i$ then $\widetilde{B}_i = \widetilde{A}_j$.  We define
a family $B = \langle B_i : i \in \N \rangle$, where $B_i = A_j$ for
the unique $j$ such that $2j \in \widetilde{B}_i$.  We claim that this
is a maximal subfamily of $A$ with the $\overline{D}_n$ intersection
property.

It is not difficult to see that~$B$ has the~$\overline{D}_n$ intersection
property.  Indeed, let $A_{i_0},\ldots,A_{j_{n-1}}$ be any~$n$ distinct
members of~$B$, and assume the indices have been chosen so that
$\widetilde{A}_{i_j} \in \widetilde{B}$ for all $j < n$.  Then
$\bigcap_{j < n} \widetilde{A}_{i_j} \neq \emp$, so by construction we
can find a finite set~$F$ of size $\geq n+1$ such that $i_j \in F$ for all~$j$
and $\bigcap_{i \in F'} A_i \neq \emp$ for every $n$-element
$F' \subset F$.  In particular, $\bigcap_{j < n} A_{i_j} \neq \emp$.

To show that~$B$ is maximal, we first argue that it is not a finite subfamily.
Assume otherwise.  Say the distinct members of~$B$ are 
$A_{i_0},\ldots,A_{i_m}$,
where the indices have been chosen so
that $\widetilde{A}_{i_j} \in \widetilde{B}$ for all $j \leq m$.  Now we
can find a finite set~$F$ of size $\geq n+1$ such that $i_j \in F$ for all~$j$
and $\bigcap_{i \in F'} A_i \neq \emp$ for every $n$-element $F' \subset F$.
If $m = 0$, this is because of our assumption on~$A$, and if $m > 0$, this is
because $\bigcap_{j \leq m} \widetilde{A}_{i_j} \neq \emp$.  Our assumption
on~$A$ also implies that the~$A_i$ for $i \in F$ cannot
form a maximal subfamily with the~$\overline{D}_n$ intersection property.
We can therefore fix a~$k$ so that $A_k \neq A_i$ for all $i \in F$ and
$\bigcap_{i \in F'} A_i \neq \emp$ for every $n$-element $F' \subset F \cup \{k\}$.
Then by construction,
$\widetilde{A}_k \cap \bigcap_{j \leq m} \widetilde{A}_{i_j} \neq \emp$.
Of course, the same is true if we replace any~$i_j$ in the intersection by any~$i$
such that $A_i = A_{i_j}$.  And since for every~$i$
such that $\widetilde{A}_i \in B$ we have $A_i = A_{i_j}$
for some $j \leq m$, it follows that the intersection of any finite
number of members of~$\widetilde{B}$ with~$\widetilde{A}_k$ is
nonempty.  By maximality of~$\widetilde{B}$, $\widetilde{A}_k \in \widetilde{B}$
and hence $A_k \in B$.  This is the desired contradiction.

Now suppose $A_k \notin B$ for some~$k$, so that necessarily
$\widetilde{A}_k \notin \widetilde{B}$.   Since~$\widetilde{B}$ is maximal,
and since~$B$ is not finite, we can consequently find a finite
set~$F$ of size $\geq n+1$ such that
\begin{itemize}
\item for all $i \neq j$ in~$F$, $A_i \neq A_j$;
\item for all $i \in F$, $\widetilde{A}_i \in \widetilde{B}$;
\item $\widetilde{A}_k \cap \bigcap_{i \in F} \widetilde{A}_i = \emp$.
\end{itemize}
By construction, this means there is an $n$-element subset~$F'$ of
$F \cup \{k\}$ with $\bigcap_{i \in F'} A_i = \emp$, and clearly~$k$ must
belong to~$F'$. Since $A_i \in B$ for all $i \in F$, and in particular for all
$i \in F' - \{k\}$, we conclude that~$B$ is maximal with respect to 
property $\overline{D}_n$.  This completes the
proof that $F\IP$ implies $\overline{D}_n\IP$.

A similar argument can be used to show (2).  We have only to modify
the construction of $\widetilde{A}$ by looking, instead of at
finite sets $F \subseteq\{0,\ldots,s\}$ with $|F| \geq n+1$, only at
those with $|F| = n+1$.  The details are left to the reader.
\end{proof}

An apparent weakness of our definition of subfamily is that we cannot,
in general, effectively decide which members of a family are contained in a given
subfamily.  The next proposition demonstrates that if we strengthen
the definition of subfamily to make this problem decidable, all the
intersection principles we study become equivalent to arithmetical
comprehension.

\begin{prop}\label{P:familydef}
Let $P$ be any of the properties in Definition~\ref{D:prop}.  
The following are equivalent over $\RCA$:
\begin{enumerate}
\item $\ACA$;

\item every nontrivial family of sets $\langle A_i : i \in \N \rangle$
has a maximal subfamily $B$ with the $P$ intersection property, and
the set $I = \{i \in \N: A_i \in B\}$ exists.

\end{enumerate}
\end{prop}

\begin{proof} 
The argument that~(1) implies~(2) is a refinement of the
proof of Proposition~\ref{prop_fip_first}.  In the case where $A$ does
not have a finite maximal subfamily with the $P$ intersection
property, we can take for $I$ the range of the function $p$ defined in
that proof.

To show that (2) implies (1), we work in $\RCAo$ and let $f\colon \N
\to \N$ be a one-to-one function.  For each $i$, let
\[
A_i = \{2i\} \cup \{2x+1 : (\exists y \leq x)[f(y) =i]\}.
\]
noting that $i \in \operatorname{range}(f)$ if and only if $A_i$ is
not a singleton, in which case $A_i$ contains cofinitely many odd
numbers.  Consequently, for every finite $F \subset \N$ of size $\geq 2$,
$\bigcap_{i \in F} A_i \neq \emp$ if and only if each $i \in F$ 
is in the range of~$f$.

Apply (2) with $P = D_n$ to the family $A = \langle A_i : i \in \N
\rangle$ to find the corresponding subfamily $B$ and set~$I$.  Because
$B$ is a maximal subfamily with the $D_n$ intersection property,
there are at most $n-1$ many $j$ such that $j \in \operatorname{range}(f)$
and $A_j \in B$.  And for each~$i$ not equal to any such~$j$, we have
\[
i \in \operatorname{range}(f) \Leftrightarrow A_i \notin B
\Leftrightarrow i \notin I.
\]
Thus the range of $f$ exists.  We reach the same conclusion if we
instead apply (2) with $P = F$ or $P = \overline{D}_n$ to~$A$.  In this case,~$B_i$
is not a singleton for all $i \in \setN$, and we have
\[ 
i \in \operatorname{range}(f) \Leftrightarrow A_i \in B
\Leftrightarrow i \in I.\qedhere
\]
\end{proof}

We close this subsection by showing that the above reversal to $\ACA$
goes through for $P = D_n$ even with our weak definition of subfamily.

\begin{prop}\label{P:DnIP_to_ACA} 
For each standard $n \geq 2$, $D_n\IP$ is equivalent 
to $\ACA$ over $\RCA$.
\end{prop}

\begin{proof} 
Fix a one-to-one function $f \colon \N \to \N$, and let $A$ be the family defined
in the preceding proposition.  Let $B = \langle B_i : i \in \N \rangle$ be the
family obtained from applying $D_n\IP$ to~$A$.  As above, there can be at most
$n-1$ many $j$ such that $j \in \operatorname{range}(f)$ and $A_j \in B$.  For~$i$
not equal to any such~$j$, we have
\[ 
i \in \operatorname{range}(f) \Leftrightarrow A_i \notin B
\Leftrightarrow (\forall k)[2i \notin B_k].
\]
This gives us a $\Pi^0_1$ definition of the range of~$f$.  Since
the range of $f$ is also definable by a $\Sigma^0_1$ formula, it
follows by $\Delta^0_1$ comprehension that the range of $f$ exists.
\end{proof}

We do not know whether the implications from $F\IP$ to
$\overline{D}_n\IP$ or from $\overline{D}_{n+1}\IP$ to
$\overline{D}_n\IP$ are strict.  However, all of our results in the
sequel hold equally well for $F\IP$ as they do for
$\overline{D}_2\IP$.  Thus, we shall formulate all implications over
$\RCA$ involving these principles as being to $F\IP$ and from
$\overline{D}_2\IP$.

\subsection{Non-implications and conservation results} In contrast
to Proposition~\ref{P:DnIP_to_ACA}, $F\IP$ and the principles
$\overline{D}_n\IP$ for $n \geq 2$ are all strictly weaker
than~$\ACA$.  This section is dedicated to a proof of this
nonimplication, as well as to results showing that $F\IP$ does not imply
$\WKLo$ and $D_2\IP$ is not provable in $\WKLo$. 
These results will be further sharpened by Proposition~\ref{P:Gen_to_FIP}
below. 

\begin{prop}\label{P:FIP_no_ACA} 
There is an $\omega$-model of $\RCA + F\IP$ consisting entirely 
of low sets.  Therefore $F\IP$ does not imply $\ACA$ over $\RCA$.
\end{prop}

\begin{proof} 
Given a computable nontrivial family $A = \langle A_i :
i \in \om \rangle$ of sets, let $\mathbb{F}_A$ be the notion of
forcing whose conditions are strings $\sigma \in \om^{<\om}$ such that
some $x \leq \sigma(|\sigma|-1)$ belongs to $A_{\sigma(i)}$ for all $i
< |\sigma| - 1$, and $\sigma' \leq \sigma$ if $\sigma' \res |\sigma'|
- 1 \succeq \sigma \res |\sigma| - 1$.  Now fix any $A_i \neq \emp$,
say with $x \in A_i$, and let $\sigma_0 = i x$.  Given $\sigma_{2e}$
for some $e \in \om$, ask if there is a condition $\sigma \leq
\sigma_{2e}$ such that $\Phi^{\sigma \res |\sigma| - 1}_e(e)
\downarrow$.  If so, let $\sigma_{2e+1}$ be the least such $\sigma$ of
length greater than $|\sigma_{2e}|$, and if not, let $\sigma_{2e+1} =
\sigma_{2e}$.  Given $\sigma_{2e+1}$, ask if there is a condition
$\sigma \leq \sigma_{2e+1}$ such that $\sigma(i) = e$ for some $i <
|\sigma| - 1$.  If so, let $\sigma_{2e+2}$ be the least such $\sigma$,
and if not, let $\sigma_{2e+2} = \sigma_{2e+1}$.  A standard argument
establishes that $J = \bigcup_{e \in \om} \left( \sigma_e \res |\sigma_e| -
1\right )$ is low, and hence so is $B = \langle A_{J(i)} : i \in \om
\rangle$.  It is clear that $B$ is a maximal subfamily of $A$ with the
$F$ intersection property.  Iterating and dovetailing this argument
produces the desired $\om$-model. 

The second part of the proposition follows from the fact that every
\mbox{$\om$-model} of $\ACA$ must contain a set of degree $\0'$, which is not low.
\end{proof}

We will establish the result that $F\IP$ does not even
imply~$\WKL$ by showing $F\IP$ is conservative for the following class
of sentences.

\begin{defn}[{Hirschfeldt, Shore and Slaman~\cite[p.~5819]{HS-2007}}]\label{D:RP} A sentence in $\Lang_2$ is \emph{restricted $\Pi^1_2$} if it is of the
form
\[
  (\forall X)[\varphi(X) \to (\exists Y)\psi(X,Y)],
\]
where $\varphi$ is arithmetical and $\psi$ is $\Sigma^0_3$.
\end{defn}

Many familiar principles are equivalent to restricted $\Pi^1_2$
sentences over $\RCAo$, including the defining
axiom of $\WKLo$. We discuss several others in the next
subsection.

The study of restricted $\Pi^1_2$ conservativity was initiated by
Hirschfeldt and Shore~\cite[Corollary 2.21]{HS-2007} in the context of
the principle $\COH$.  Subsequently, it was extended by Hirschfeldt,
Shore, and Slaman~\cite[Corollary 3.15 and the penultimate paragraph
of Section 4]{HSS-2009} to the principles $\AMT$ and
$\Pi^0_1\mathsf{G}$ (see Definitions~\ref{D:AMT} and~\ref{D:Pi01G}
below).  The conservation proofs for the latter two principles differ
from the original only in the choice of forcing notion (Mathias
forcing for $\COH$, Cohen forcing for $\AMT$ and $\Pi^0_1\mathsf{G}$).
A similar proof goes through, \emph{mutatis mutandis}, for the
notion $\mathbb{F}_A$ from the proof of
Proposition~\ref{P:FIP_no_ACA}, giving the following conservation
result.  We refer the reader to either of the above-cited papers for
details.

\begin{thm}\label{P:FIP_conservative} The principle $F\IP$ is
conservative over $\RCA$ for restricted $\Pi^1_2$ sentences.
Therefore $F\IP$ does not imply $\WKL$ over $\RCA$.
\end{thm}

The preceding results lead to the question of whether $F\IP$, or any
one of the principles $\overline{D}_n\IP$, is provable in $\RCA$, or
at least in $\WKL$.  We show in the following theorem
that $F\IP$ fails in any $\om$-model of $\WKL$ consisting entirely of sets
of hyperimmune-free Turing degree.  Recall that a Turing degree is
\emph{hyperimmune} if it bounds the degree of a function not dominated
by any computable function, and a degree which is not hyperimmune is
\emph{hyperimmune-free}.  A model of the kind we are interested in can
be obtained by iterating and dovetailing the hyperimmune-free basis
theorem of Jockusch and Soare~\cite[Theorem 2.4]{JS-1972}, which
asserts that every infinite computable subtree of $2^{<\om}$ has an
infinite path of hyperimmune-free degree.

\begin{thm}\label{T:FIP_to_OPT}
There exists a computable nontrivial
family of sets for which any maximal subfamily with the $\overline{D}_2$
intersection property must have hyperimmune degree.
\end{thm}

To motivate the proof, which will occupy the rest of this subsection,
we first discuss the simpler construction of a computable
nontrivial family for which any maximal subfamiliy
with the $\overline{D}_2$ intersection property must be noncomputable.  This, in turn, is perhaps
best motivated by thinking how a proof of the contrary could fail.

Suppose we are given a computable nontrivial family \mbox{$A = \langle
A_i : i \in \N \rangle$.}  The most direct method of building a
maximal subfamily $B = \langle B_i : i \in \N \rangle$ with the
$\overline{D}_2$ intersection property, assuming $A$ has no finite
such subfamily, is to let $B_0 = A_i$ for the least $i$ so that $A_i
\neq \emp$, then to let $B_1 = A_j$ for the least $j > i$ such that
$A_i \cap A_j \neq \emp$, and so on.  Of course, this subfamily will
in general not be computable, but we could try to temper our strategy
to make it computable.  An obvious such attempt is the following.  We
first search through the members of $A$ in some effective fashion
until we find the first one that is nonempty, and we let this be
$B_0$.  Then, having defined $B_0,\ldots,B_n$ for some $n$, we search
through $A$ again until we find the first member not among the $B_i$
but intersecting each of them, and let this be $B_{n+1}$.  Now while
this strategy yields a subfamily $B$ which is indeed computable and
has the $\overline{D}_2$ intersection property, $B$ need not be
maximal.  For example, suppose the first nonempty set we discover is
$A_1$, so that we set $B_0 = A_1$.  It may be that $A_0$ intersects
$A_1$, but that we discover this only after discovering that $A_2$
intersects $A_1$, so that we set $B_1 = A_2$.  It may then be that
$A_0$ also intersects $A_2$, but that we discover this only after
discovering that $A_3$ intersects $A_1$ and $A_2$, so that we set $B_2
= A_3$.  In this fashion, it is possible for us to never put $A_0$
into $B$, even though it ends up intersecting each $B_i$.

We can exploit precisely this difficulty to build a family $A =
\langle A_i : i \in \om \rangle$ for which neither the strategy above,
nor any other computable strategy, succeeds.  We proceed by stages, at
each one enumerating at most finitely many numbers into at most
finitely many $A_i$.  By Remark~\ref{R:FIP_subfamilies}, it suffices
to ensure that for each $e$, either $\Phi_e$ is not total, or else
$\langle A_{\Phi_e(i)} : i \in \om \rangle$ is not a maximal subfamily
with the $\overline{D}_2$ intersection property.  We discuss how to
satisfy a single such requirement.  Of course, in the full
construction there will be other requirements, but these will not
interfere with one another.

At stage $s$, we look for the longest nonempty string $\sigma \in \om^{<\om}$
such that for all $i < |\sigma|$, $\Phi_e(i)[s] \downarrow =
\sigma(i)$, and for all $i,j < |\sigma|$, $A_{\sigma(i)}$ and
$A_{\sigma(j)}$ have been intersected by stage~$s$.  At the first
stage that we find such a $\sigma$, we define $t_e$ to be some number
large enough that $A_{t_e}$ does not yet intersect $A_{\Phi_e(i)}$ for
any~$i$.  We then start defining numbers $p_{e,0}, p_{e,1},\ldots$ as
follows.  At each stage, if we do not find a longer such $\sigma$, or if
$t_e$ is in the range of this $\sigma$, we do nothing.  Otherwise,
we choose the least $n$ such that $p_{e,n}$ has not yet been defined,
and define it be some number not yet in the range of $\Phi_e$ and large
enough that $A_{p_{e,n}}$ does not intersect $A_{t_e}$.  We call $p_{e,n}$
a \emph{follower} for $\sigma$.  Then for any $p_{e,m}$ that is already
defined and is a follower for some $\tau \preceq \sigma$, we intersect
$A_{p_{e,m}}$ with $A_{\sigma(i)}$ for all~$i$.  Also, if $\sigma(i) = p_{e,m}$
for some $i$ and $m$, then for the largest such $m$ and for all $j$ with
$\sigma(j) \neq p_{e,m}$, we intersect $A_{\sigma(j)}$ with $A_{t_e}$.

Now suppose that $\Phi_e$ is total and that the subfamily it defines
is a maximal one with the $\overline{D}_2$ intersection property.  The
idea is that $A_{t_e}$ should behave as $A_0$ did in the motivating
example above,  by never entering the subfamily but intersecting all
of its members, thereby giving us a contradiction. For the first part,
note that if $\Phi_e(i) = t_e$ for some $i$ then $\sigma(i) = t_e$ for
some string $\sigma$ as above, and that necessarily $A_{\sigma(j)}
\cap A_{t_e} = \emp$ for some~$j$.  But any string we find at a
subsequent stage will extend $\sigma$ and hence have $t_e$ in its
range, so we will never make $A_{\sigma(j)} = A_{\Phi_e(j)}$ intersect
$A_{t_e} = A_{\Phi_e(i)}$.  Thus, $t_e$ cannot be the range of
$\Phi_e$.  We conclude that $p_{e,n}$ is defined for every~$n$.  For
the second part, note that since each $p_{e,n}$ is a follower for some
initial segment of $\Phi_e$, each $A_{\Phi_e(i)}$ is eventually
intersected with $A_{p_{e,n}}$.  By maximality, then, $p_{e,n}$
belongs to the range of $\Phi_e$ for all $n$, which means that each
$A_{\Phi_e(i)}$ is eventually also intersected with $A_{t_e}$.

This basic idea is the same one that we now use in our proof of
Theorem~\ref{T:FIP_to_OPT}.  But since here we are concerned with
more than just computable subfamilies, it no longer suffices to just play
against those of the form $\langle A_{\Phi_e(i)} : i \in \om \rangle$.
Instead, we must consider all possible subfamilies $\langle A_{J(i)} :
i \in \om \rangle$ for $J \in \om^{\om}$, and show that if $J$ defines
a maximal subfamily with the $\overline{D}_2$ intersection
property then there exists a function $f \leq_T J$ such that for all~$e$,
either $\Phi_e$ is not total or it does not dominate~$f$.
Accordingly, we must now define followers $p_{e,n}$ not only for
those $\sigma \in \om^{<\om}$
that are initial segments of $\Phi_e$, but for all strings that look as
though they can be extended to some such $J \in \om^\om$.
We still enumerate
the followers linearly as
$p_{e,0}, p_{e,1}, \ldots$, even though the strings they are defined as
followers for no longer have to be compatible.

Looking ahead to the verification, fix any $J$ that defines
a maximal subfamily with the $\overline{D}_2$ intersection
property.  We describe
the intuition behind defining $f \leq_T J$ that escapes domination by a single
computable function $\Phi_e$.  (Of course, there
are much easier ways to define $f$ to achieve this, but this definition is
close to the one that will be used in the full construction.)
Much as in the more basic argument above,
the construction will ensure that there are infinitely many~$n$ such that $p_{e,n}$ is 
a follower for some initial segment of~$J$ and belongs to the range
of~$J$.  Then, $f(x)$ can be thought of as telling us how far to go along~$J$ in order to find
one more $p_{e,n}$ in its range.  More precisely,~$f$ is defined along
with a sequence $\sigma_0 \prec \sigma_1 \prec \cdots$ of initial segments
of~$J$.  For each~$x$, $\sigma_{x+1}$ is an extension of $\sigma_x$ whose
range contains a follower $p_{e,n}$ for some $\tau$
with $\sigma_x \preceq \tau \prec \sigma_{x+1}$, and $f(x+1)$ is a number large
enough to bound an element of
$\bigcap_{i < |\sigma_{x+1}|} A_{\sigma_{x+1}(i)}$.
The idea behind this definition is that if $f$ actually is dominated by $\Phi_e$,
then we can modify our basic strategy so that in deciding which members of
$A$ to intersect with $A_{t_e}$ in the construction, we consider not initial segments
of~$\Phi_e$ as before, but strings $\sigma \in \om^{<\om}$ that look like initial
segments $J$.  Then, just as before,
we can show that no such string $\sigma$ can have $t_e$ in its range, and yet
that $A_{\sigma(i)}$ is eventually intersected with $A_{t_e}$ for all $i$.
Thus we obtain the same contradiction we got above, namely that $J$ does not
have $t_e$ in its range and hence cannot be maximal after all.

The main obstacle to this approach is that we do not know which computable function
will dominate $f$, if $f$ is in fact computably dominated, and so we cannot use its index in the
definition of $f$.  One way to remedy this is to make $f(x)$ large enough to
find not only the next $p_{e,n}$ in the range of $J$ for some fixed $e$, but the
next $p_{e,n}$ for each $e < x$.  This, in turn, demands that we define followers in
such a way that $p_{e,n}$ is defined for every~$e$ and~$n$, regardless of whether
$\Phi_e$ is total.  But then we must define followers $p_{e,n}$ even for
strings that already contain $t_e$ in their range, since we do not know ahead of time
that this will not happen for all sufficiently long strings.  In the construction, then, we distinguish
between two types of followers, those defined as followers for strings that have $t_e$
in their range, and those defined as followers for strings that do not.  We will see in
the verification that we can restrict ourselves to strings of the latter type, so this is
not a serious complication.

We turn to the formal details.  We adopt the convention that for all $e,x,y,s \in \om$,
if $\Phi_e(x)[s] \downarrow = y$, then $e,x,y \leq s$, and
$\Phi_e(z)[s] \downarrow$ for all $z < x$.  Let $s_{e,x}$ denote the
least $s$ such that $\Phi_e(x)[s] \downarrow$, which may of course be
undefined if $\Phi_e$ is not total.  Then to show that some function is not
computably dominated it suffices to show it is not dominated by the map
$x \mapsto s_{e,x}$ for any $e$.

\begin{proof}[Proof of Theorem~\ref{T:FIP_to_OPT}] 
We build a computable $A = \langle A_i : i \in \om \rangle$ by stages.  
Let $A_i[s]$ be the set of elements which have been enumerated into $A_i$
by stage $s$, which will always be finite. Say a nonempty string
$\sigma \in \om^{<\om}$ is {\em bounded} by $s$ if
\begin{itemize}
\item $|\sigma| \leq s$;

\item for all $i < |\sigma|$, $\sigma(i) \leq s$;

\item for all $i,j < |\sigma|$, there is a $y \leq s$ with $y \in
A_{\sigma(i)}[s] \cap A_{\sigma(j)}[s]$.
\end{itemize} 

\medskip
\noindent {\em Construction.} For all $i \neq j$, let $2i \in A_i$ and
$2j \notin A_i$.  At stage $s \in \om$, assume inductively that for
each $e$, we have defined finitely many numbers $p_{e,n}$, $n \in
\om$, each labeled as either a {\em type 1 follower} or a {\em type 2
follower} for some string $\sigma \in \om^{<\om}$.  Call a number $x$
{\em fresh} if $x$ is larger than $s$ and every number that has been
mentioned during the construction so far.

We consider consecutive substages, at substage $e \leq s$ proceeding
as follows.

\medskip
\noindent {\em Step 1.} If $t_e$ is undefined, define it to be a fresh
number.  If $t_e$ is defined but $s_{e,0} = s$, redefine $t_e$ to be a
fresh large number.  In the latter case, change any type 1 follower
$p_{e,n}$ already defined to be a type 2 follower (for the same string).

\medskip
\noindent {\em Step 2.} Consider any $\sigma \in \om^{<\om}$ bounded
by~$s$.  Choose the least $n$ such that $p_{e,n}$ has not been
defined, and define it to be a fresh number.  Then, for each $i <
|\sigma|$, enumerate a fresh odd number into $A_{p_{e,n}} \cap
A_{\sigma(i)}$.  If there is an $i < |\sigma|$ such that $\sigma(i) =
t_e$, call $p_{e,n}$ a type 1 follower for $\sigma$, and otherwise,
call $p_{e,n}$ a type 2 follower for~$\sigma$.

\medskip
\noindent {\em Step 3.} Consider any $p_{e,n}$ defined at a stage
before $s$, and any $\sigma \in \om^{<\om}$ bounded by $s$ that
extends the string that $p_{e,n}$ was defined as a follower for. If
$p_{e,n}$ is a type 1 follower then, for each $i < |\sigma|$,
enumerate a fresh odd number into $A_{p_{e,n}} \cap A_{\sigma(i)}$.
If $p_{e,n}$ is a type 2 follower, then do this only for the $\sigma$
such that $\sigma(i) \neq t_e$ for all~$i$.

\medskip
\noindent {\em Step 4.}  Suppose there is an $x$ such that
$\Phi_e(x)[s] \downarrow$, and $s = s_{e,x}$ for the largest such~$x$.
Call a string $\sigma \in \om^{<\om}$ {\em viable} for $e$ at stage
$s$ if there exist $\sigma_0 \prec \cdots \prec \sigma_x = \sigma$
satisfying
\begin{itemize}
\item $|\sigma_0| = 1$;
\item for each $i \leq x$, $\sigma_i$ is bounded by $s_{e,i}$;
\item for each $i < x$ and $j \leq i$, there exists a $k$ with
$|\sigma_i| \leq k < |\sigma_{i+1}|$ and an $n$ such that $p_{j,n}$ is
defined and is a follower for some $\tau$ with $\sigma_i \preceq \tau
\prec \sigma_{i+1}$, and $\sigma_{i+1}(k) = p_{j,n}$.
\end{itemize} If $x > e$, let $k^\sigma_{x,e}$ be the least $k$ that
satisfies the last condition above for $i = x-1$ and $j = e$.

Call $s$ an {\em $e$-acceptable} stage if for every string $\sigma$
viable for $e$ at this stage,
\begin{itemize}
\item $k^\sigma_{e,x}$ is defined;

\item $A_{\sigma(k^\sigma_{e,x})}[s] \cap A_{t_e}[s] = \emp$.

\item there is an $i < k^\sigma_{e,x}$ such that

\begin{itemize}
  \item $\sigma(i) = p_{e,n}$ for some $n$;
  \item $A_{\sigma(i)}[s] \cap A_{t_e}[s] = \emp$;
  \item for all $j \leq i$ and all $\tau$ viable for $e$ at stage $s$,
        $\sigma(j) \neq \tau(k^\tau_{e,x})$.
\end{itemize}

\end{itemize} 
If $s$ is $e$-acceptable, then for each viable
$\sigma$, choose the largest such $i < k^\sigma_{x,e}$, and enumerate
a fresh odd number into $A_{\sigma(j)} \cap A_{t_e}$ for each $j \leq
i$.

\medskip
\noindent {\em Step 5.} If $e< s$, go to the next substage.  If $e <
s$, then for each $i$ and each $x$ less than or equal to the largest
number mentioned during the construction at stage $s$ and and not
enumerated into $A_i$, enumerate $x$ into the complement of $A_i$.
Then go to stage $s+1$.

\medskip
\noindent {\em End construction.}

\medskip
\noindent {\em Verification.}  It is clear that $A$ is a computable
nontrivial family.  Suppose $B = \langle B_i : i \in \N \rangle$ is a
maximal subfamily of $A$ with the $\overline{D}_2$ intersection
property.  Choose the unique $J \in \om^\om$ such that $B_i =
A_{J(i)}$ for all~$i$.

\begin{claim} 
For each $e \in \om$ and each $\sigma \prec J$, there is
an $n \in \om$ such that $p_{e,n}$ is a follower for some $\tau$ with
$\sigma \preceq \tau \prec J$ and $A_{p_{e,n}} \in B$.
\end{claim}

\begin{proof} 
First, notice that for each $\sigma \preceq J$, there
are infinitely many $s$ that bound~$\sigma$.  Hence, since at any such
stage $s$ of the construction (specifically, at step~2 of substage
$e$), $p_{e,n}$ gets defined for a new $n \in \om$, it follows that
$p_{e,n}$ gets defined for all~$n$.  Second, note that $t_e$
necessarily gets defined during the construction, and then gets
redefined at most once.  We use $t_e$ henceforth to refer to its final
value.

Fix $\sigma \prec J$ and $m \in \om$, and let $s$ be a stage by which
$p_{e,n}$ has been defined for all $n \leq m$.  Let $\tau$ be either
$\sigma$ if $A_{t_e} \notin B$ or $\sigma(i) = t_e$ for some $i <
|\sigma|$, or an initial segment of $J$ extending $\sigma$ long enough
that there exists a $i < | \tau|$ with $\tau(i) = t_e$.  By our
observation above, there exists a $t \geq \max\{s,e\}$ that bounds
$\tau$.  Let $p_{e,n}$ be the follower for $\tau$ defined at stage
$t$, substage $e$, step~2, of the construction, so that necessarily $n
> m$.  Note that $p_{e,n}$ is a type 2 follower if and only if
$A_{t_e} \notin B$.

Choose any $\upsilon$ with $\tau \preceq \upsilon \prec J$, and let $u
> t$ be large enough to bound~$\upsilon$.  Then at stage $u$, substage
$e$, step~3, of the construction, $A_{p_{e,n}}$ is made to intersect
$A_{\upsilon(i)}$ for each $i < | \upsilon |$ (in case $p_{e,n}$ is a
type 2 follower, this is because $\upsilon(i) \neq t_e$ for all $i$).
Since $\upsilon $ was arbitrary, it follows that $A_{p_{e,n}} \cap
A_{J(i)}$ for all $i \in \om$.  Hence, by maximality of $B$, it must
be that $A_{p_{e,n}} \in B$.
\end{proof}

Now define a function $f : \N \to \N$, and a sequence $\sigma_0 \prec
\sigma_1 \prec \cdots$ of initial segments of $J$, as follows.  Let
$\sigma_0 = J \res 1$ and $f(0) = 2J(0)$, and assume that we have
$f(x)$ and $\sigma_x$ defined for some $x \geq 0$.  Let $f(x+1)$ be
the least $s$ such that there exists a $\sigma \in \om^{<\om}$
satisfying
\begin{itemize}
\item $\sigma_x \prec \sigma \prec J$;

\item $\sigma$ is bounded by $s$;

\item for each $j \leq x$, there exists a $k$ with $|\sigma_x| \leq k
< |\sigma|$ and an $n$ such that $p_{j,n}$ is defined by stage $s$ of
the construction and is a follower for some $\tau$ with $\sigma_x
\preceq \tau \prec \sigma$, and $\sigma(k) = p_{j,n}$.

\end{itemize} 
Let $\sigma_{x+1}$ be the least $\sigma$ satisfying the
above conditions.  By the preceding claim, $f(x)$ and $\sigma_x$ are
defined for all~$x$.

Clearly, $f \leq_T B$.  Seeking a contradiction, suppose $e$ is such that $f(x)
\leq s_{e,x}$ for all~$x$.  A simple induction then shows that $\sigma_x$
is viable for $e$ at stage $s_{e,x}$.  So in particular, for every $x$, there is
a $\sigma$ viable for $e$ at stage $s_{e,x}$.  We fix the present value
of $e$ for the remainder of the proof, including in the following claims.

\begin{claim} 
If $\sigma$ is viable for $e$ at stage $s_{e,0}$, then
$A_{\sigma(0)}$ is not intersected with $A_{t_e}$ before step 4 of
substage $e$ of the first $e$-acceptable stage.
\end{claim}

\begin{proof} 
Note that necessarily $|\sigma| = 1$, and that $s_{e,0}$
is not $e$-acceptable.  At step 1 of substage $e$ of stage
$s_{e,0}$, $t_e$ gets redefined to be a fresh number.  Viability at
stage $s_{e,0}$ just means that $\sigma$ is bounded by $s_{e,0}$,
and hence $A_{\sigma(0)}$ cannot intersect $A_{t_e}$ at the end
of this step.  Hence, if we let $s$ be the stage at which $A_{\sigma(0)}$
is first intersected with
$A_{t_e}$, then $s \geq s_{e,0}$.  Suppose the intersection takes
place at step $k$ of substage $i$ of stage~$s$.  Then in particular,
this point in the construction comes strictly after step 4 of substage
$e$ of stage $s_{e,0}$.

It suffices to prove the claim under the following assumption:
there is no $\sigma'$ viable for $e$ at stage $s_{e,0}$ such that
$A_{\sigma(0)}$ is first intersected with $A_{t_e}$ before step $k$ of
substage $i$ of stage~$s$.  Note also that $k$ must be 3 or 4, since
the only other step at which different members of $A$ are intersected
is step~2, but one of the two sets intersected there is always indexed
by a fresh number.

First suppose $k = 3$.  Then it must be that for some~$n$, and for
some~$\tau$ extending the string~$\rho$ for which $p_{i,n}$ is a follower,
we are intersecting~$A_{p_{i,n}}$ with~$A_{\tau(i)}$ for all $i < |\tau|$.
Since~$t_e$ cannot equal~$p_{i,m}$ for any~$m$,
it must be that $\sigma(0) = p_{i,n}$, and hence that there
is a a $j < |\tau|$ such that $\tau(j) = t_e$.  Now $\sigma(0)$ is bounded by
$s_{e,0}$ and hence is not fresh after step 4 of substage $e$ of stage~$s_{e,0}$,
whereas $p_{i,n}$, when defined, is defined to be a fresh number.  Thus, since
$\sigma(0) = p_{i,n}$, $p_{i,n}$ must be defined as a follower for $\rho$ before
step 4 of substage $e$ of stage $s_{e,0}$.  At that point in the construction, by definition,
$\rho$ has to be bounded, so $\rho$ must also be bounded by $s_{e,0}$.
In particular, $\rho(j)$ must be viable for $e$ at stage $s_{e,0}$, for every $j$.
This means $\rho(j) \neq t_e$, since $t_e$ is certainly not viable at stage $s_{e,0}$.
But since $\tau$ has to be bounded by $s_{e,0}$ in order for us to be considering
it, it must be that $A_{\rho(j)}$ and $A_{t_e}$ are intersected at some earlier point
in the construction.  This contradicts our assumption above.

Now suppose $k = 4$ but $i \neq e$.  Then it must be that $s$ is
$i$-acceptable.  Since $t_e$ cannot equal $t_i$, and since members
of $A$ are only intersected at step 4 with $A_{t_i}$, it must be that
$\sigma(0) = t_i$.  There must also be a $\tau \in \om^{<\om}$ such
that $\tau$ is viable for $i$ at stage $s$ and $\tau(j) = t_e$ for
some $j < |\tau|$.  Since $s$ is $i$-acceptable, $s_{i,0}$ is defined.
Now $\sigma(0)$ is bounded by
$s_{e,0}$ and hence is not fresh after step 4 of substage $e$ of
stage~$s_{e,0}$, whereas at step 1 of substage $i$ of stage
$s_{i,0}$, $t_i$ is redefined to be a fresh number.  Thus, since
$\sigma(0) = t_i$, step 1 of substage $i$ of stage $s_{i,0}$ cannot
happen after step 4 of substage $e$ of stage $s_{e,0}$.  So,
since the one bit string $\tau(0)$ has to be viable for $i$ at step $s_{i,0}$
by definition of viability, it follows that $\tau(0)$ is also viable at stage
$s_{e,0}$.  Hence, $\tau(0) \neq t_e$ since $t_e$ is not viable at stage
$s_{e,0}$.  But since $\tau$ has to be bounded by $s_{e,0}$, it must be
that $A_{\tau(0)}$ and $A_{t_e}$ are intersected at some earlier point
in the construction.  This again gives us a contradiction.

We conclude that $k = 4$ and $i = e$, that is, that $A_{\sigma(0)}$ is
first intersected with $A_{t_e}$ at step~4 of substage $e$ of
stage~$s$.
This forces $s$ to be $e$-acceptable, so the claim is proved.
\end{proof}

\begin{claim}\label{C:fiphyp1} 
Suppose $x > e$ and $\sigma \in
\om^{<\om}$ is viable for $e$ at stage $s_{e,x}$.  Then for some $i <
|\sigma|$, $\sigma(i) = p_{e,n}$ for some $n$ and $A_{\sigma(i)}$ and
$A_{t_e}$ are disjoint through the end of stage~$s_{e,x}$.
\end{claim}

\begin{proof} 
We proceed by induction on $x$, beginning with $x =
e+1$.  Fix~$\sigma$.  By construction, $s_{e,x}$ is the first stage
that can be $e$-acceptable, so by the preceding claim, $A_{\sigma
(0)}$ has empty intersection with $A_{t_e}$ at the beginning of
step 4 of substage $e$ of this stage.  Hence, $\sigma(i) \neq t_e$ for
all $i < |\sigma|$ since $\sigma$ must be bounded by $s_{e,x}$.
Now by viability, there is an $i$ and an $n$ such
that $\sigma(i) = p_{e,n}$ and is a follower for some $\tau$ with
$\sigma(0) \preceq \tau \prec \sigma$.  It follows that $p_{e,n}$
is a type 2 follower.  Furthermore, it is easy to see that for any type 2
follower $p_{e,m}$, $A_{p_{e,m}}$ can only be made to intersect
$A_{t_e}$ at step 4 of substage $e$ of an $e$-acceptable stage.
Thus, $A_{\sigma(i)}$ must
be disjoint from $A_{t_e}$ at the beginning of step 4 of substage $e$
of stage $s_{e,x}$.  Additionally, if $s_{e,x}$ is not
$e$-acceptable, then nothing is done at step 4 of substage $e$, and
hence $A_{\sigma(i)}$ is not intersected with $A_{t_e}$ during the
course of the rest of the stage.  If $s_{e,x}$ is $e$-acceptable,
then in fact there must exist a $i$ as above, namely $i =
k^\sigma_{x,e}$, such that $A_{\sigma(i)}$ is not intersected with
$A_{t_e}$ at step 4 of substage $e$, and hence not during the course
of the rest of the stage either.  This proves the base case of the
induction.

Now let $x > e$ be given and suppose the claim holds for~$x$.  Given
$\sigma \in \om^{<\om}$ viable for $e$ at stage $s_{e,x+1}$, there is
some $\tau \prec \sigma$ viable for $e$ at stage $s_{e,x}$.  If
$s_{e,x+1}$ is not $e$-acceptable, then the same $i$ witnessing that
the claim holds for $x$ and $\tau$ witnesses also that it holds for
$x+1$ and~$\sigma$.  This is because $\tau(i)$ is necessarily a type 2
follower, and $A_{\tau(i)}$ is consequently not intersected with
$A_{t_e}$ until step 4 of substage $e$ of some $e$-acceptable stage
after stage~$s_{e,x}$.  If $s_{e,x+1}$ is $e$-acceptable, then just
as in the base case, viability of $\sigma$ implies that for $i =
k^\sigma_{x+1,e}$, $A_{\sigma(i)}$ does not intersect $A_{t_e}$ at the
beginning of step 4 of substage $e$ of stage $s_{e,x+1}$, and is not
made to do so by its end.
\end{proof}

\begin{claim}\label{C:fiphyp2} 
There exist infinitely many $e$-acceptable stages.
\end{claim}

\begin{proof} 
Fix any stage $s = s_{e,x}$ for $x > e$, and assume
there is not any  $e$-acceptable stage greater than~$s$.  For each
$\sigma$ viable for $e$ at stage $s$, let $i_{\sigma}$ be the largest
$i$ satisfying the statement of the preceding claim.  Then
$\sigma(i_\sigma)$ is a type 2 follower, so by our assumption,
$A_{\sigma(i_\sigma)}$ is never intersected with $A_{t_e}$ during the
course of the rest of the construction.

Now for each $y \geq x$ and each $\sigma$ viable for $e$ at stage
$s_{e,y+1}$, $k^\sigma_{e,y+1}$ is defined and
$\sigma(k^\sigma_{e,y+1})$ is a follower $p_{e,n}$ for some string
extending a $\tau \prec \sigma$ viable for $e$ at stage
$s_{e,y}$.  Since followers are
always defined to be fresh numbers, if $k^\tau_{e,y}$ is defined then
$\sigma(k^\tau_{e,y}) = p_{e,m}$ for some $p_{e,m}$ defined strictly
before $p_{e,n}$ in the construction.

Thus, for any sufficiently large $y > x$, it must be that for each
$\sigma$ viable at stage $s_{e,y}$, $\sigma(k^\sigma_{e,y}) \neq
\tau(k)$ for all $\tau$ viable at stage $s$ and all $k \leq j_\tau$.
Moreover, since $A_{\tau(i_\tau)} \cap A_{t_e} = \emp$ and
$\sigma(k^\sigma_{e,y})$ is a follower for some extension of some such
$\tau$, it must be that $\sigma(k^\sigma_{e,y})$ is a type 2 follower.
Hence, $A_{\sigma(k^\sigma_{e,y})}$ can only be intersected with
$A_{t_e}$ at step 4 of substage $e$ of an $e$-acceptable stage,
meaning at a stage at or before~$s$.  It follows that if $y$ is
additionally chosen large enough that, for each $\sigma$ viable at
stage $s_{e,y}$, the follower $\sigma(k^\sigma_{e,y})$ is not defined
before stage $s$, then $A_{\sigma(k^\sigma_{e,y})}$ will be disjoint
from $A_{t_e}$.  But then in particular,
$A_{\sigma(k^\sigma_{e,y})}[s_{e,y}] \cap A_{t_e}[s_{e,y}] = \emp$, so
$s_{e,y}$ is an $e$-acceptable stage greater than~$s$.  This is a
contradiction, so the claim is proved.
\end{proof}

We can now complete the proof.  First note that $A_{t_e} \notin B$,
for otherwise there would have to be an $x$ and an $i < |\sigma_x|$
such that $\sigma_x(i) = t_e$.  But then $\sigma_x$ would be viable
for $e$ at stage $s = s_{e,x}$, and so is in particular it would be
bounded by $s$, meaning $A_{\sigma_x(j)}[s]$ would have to intersect
$A_{\sigma_x(i)}[s] = A_{t_e}[s]$ for all $j < |\sigma_x|$.  This
would contradict Claim~\ref{C:fiphyp1}.  Now consider any
$e$-acceptable stage $s = s_{e,x}$.  By construction, there is an $i
< |\sigma_x|$ such that $A_{\sigma_x(i)}$ is disjoint from $A_{t_e}$
at the beginning of stage $s$, and each $A_{\sigma_x(j)}$ for $j \leq
i$ is made to intersect $A_{t_e}$ by the end of stage~$s$.  Since, by
Claim~\ref{C:fiphyp2}, there are infinitely many $e$-acceptable
stages, and since $J = \bigcup_x \sigma_x$, it follows that $A_{J(i)}$
intersects $A_{t_e}$ for all~$i$.  In other words, $B_i$ intersects
$A_{t_e}$ for all $i$, which contradicts the choice of $B$ as a
maximal subfamily of $A$ with the $\overline{D}_2$ intersection
property.
\end{proof}

\begin{rem}\label{R:opt_thm_formal}
Examination of the above proof shows that it can be  formalized in
$\RCAo$, because the construction is computable and the verification
that the function $f$ defined in it is total requires only $\Sigma^0_1$
induction.  (See \cite[Definition VII.1.4]{Simpson-2009} for
the formalizations of Turing reducibility and equivalence in $\RCA$.)
\end{rem}

\noindent As discussed above, this has as a consequence the following corollary.

\begin{cor}\label{C:FIP_no_WKL} 
The principle $\overline{D}_2\IP$ is not provable in $\WKL$.
\end{cor}
\begin{proof}
  Let $\M$ be an $\omega$-model of $\WKLo$ such that every
  set in $\M$ is of hyperimmune-free degree.  Let $A$ be the
  family constructed by the formalized version of
  Theorem~\ref{T:FIP_to_OPT}, noting that $A$ belongs to
  $\mathsf{REC}$ and hence to $\M$.  Suppose $B \in \M$ is a
  maximal subfamily of $A$ with the $\overline{D}_2$
  intersection property.  Then by the preceding remark, $\M
  \models \text{``}B \text{ has hyperimmune degree''}$. Now
  the property of having hyperimmune degree is defined by an
  arithmetical formula, and is thus absolute to
  $\omega$-models. Therefore, $B$ has hyperimmune degree,
  contradicting the construction of~$\M$.
\end{proof}

\subsection{Relationships with other principles}

By the preceding results, $F\IP$ and the principles $D_n\IP$ are of
the irregular variety that do not admit reversals to any of the main
subsystems of $\mathsf{Z}_2$ mentioned in the introduction.  In
particular, they lie strictly between $\RCA$ and $\ACA$, and are
incomparable with $\WKL$. Many principles of this kind have been
studied in the literature, and collectively they form a rich and
complicated structure. Partial summaries are given by Hirschfeldt and
Shore~\cite[p.~199]{HS-2007} and Dzhafarov and Hirst~\cite[p.
150]{DH-2009}. Additional discussion of the princples is given by 
Montalb\'{a}n~\cite[Section 1]{Montalban-TA} and Shore~\cite{Shore-2010}. 
 In this subsection, we investigate where
our intersection principles fit into the known collection of irregular
principles.

We can already show that $F\IP$ does not imply Ramsey's theorem for
pairs ($\RT^2_2$) or any of of the main combinatorial principles
studied by Hirschfeldt and Shore~\cite{HS-2007} (all of which follow
from $\RT^2_2$).  See~\cite[Definition 3.2]{DH-2009} for a concise
list of definitions of the principles in the following corollary.

\begin{cor}
None of the following principles are implied by $F\IP$
over $\RCA$: $\RT^2_2$, $\SRT^2_2$, $\DNR$, $\CAC$, $\ADS$, $\SADS$,
$\COH$.
\end{cor}

\begin{proof}
  All but the last of these principles are equivalent to restricted
  $\Pi^1_2$ sentences, and so for them the corollary follows by the
  conservation result of Proposition~\ref{P:FIP_conservative}.  For
  $\COH$, it follows by Proposition~\ref{P:FIP_no_ACA} and the fact
  that any $\om$-model of $\COH$ must contain a set of $p$-cohesive
  degree~\cite[p.~27]{CJS-2001}, and such degrees are never
  low~\cite[Theorem~2.1]{JS-1993}.
\end{proof}

Our next results require several basic model-theoretic concepts. 
 We assume some suitable development of model theory in
$\RCA$ (compare \cite[Section II.8]{Simpson-2009}).  Let $T$ be a
countable, complete, consistent theory.
\begin{itemize}
\item A \emph{partial type} of $T$ is a $T$-consistent set of formulas
in a fixed number of free variables.  A \emph{complete type}
 is a $\subseteq$-maximal partial type. 

\item A model $\M$ of $T$ \emph{realizes} a partial type $\Gamma$ if
there is a tuple $\vec{a} \in |\M|$ such that $\M \models
\varphi(\vec{a})$ for every $\varphi \in \Gamma$.  Otherwise, $\M$
\emph{omits}~$\Gamma$.

\item A partial type $\Gamma$ is \emph{principal} if there is a
formula $\varphi$ such that \mbox{$T \vdash \varphi \to \psi$} for every
formula $\psi \in \Gamma$.  A model $\M$ of $T$ is \emph{atomic} if
every partial type realized in $\M$ is principal.

\item An \emph{atom} of $T$ is a formula $\varphi$ such that for every
formula $\psi$ in the same free variables, exactly one of $T \vdash
\varphi \to \psi$ or $T \vdash \varphi \to \lnot \psi$ holds.  $T$ is
\emph{atomic} if for every $T$-consistent formula $\psi$, $T \vdash
\varphi \to \psi$ for some atom~$\varphi$.

\end{itemize} 
A classical result states that a theory is atomic if and
only if it has an atomic model.  Hirschfeldt, Slaman, and Shore
\cite{HSS-2009} studied the strength of this theorem in the following
forms.

\begin{defn}[{\cite[pp.~5808,~5831]{HSS-2009}}]\label{D:AMT} 
The following principles are defined in $\RCA$.

\begin{list}{\labelitemi}{\leftmargin=0em}\itemsep2pt
\item[]($\AMT$) Every complete atomic theory has an atomic model.
\medskip

\item[]($\OPT$) Let $T$ be a complete theory and let $S$ be a set of
partial types of~$T$. Then there is a model of $T$ omitting all the
nonprincipal partial types in~$S$.
\end{list}
\end{defn}

Over $\RCA$, $\AMT$ is strictly implied by $\SADS$ (\cite[Corollary
3.12 and Theorem 4.1]{HSS-2009}).  The latter asserts that every
linear order of type $\om + \om^*$ has a suborder of type $\om$ or
$\om^*$, and is one of the weakest principles studied in
\cite{HS-2007} that does not hold in the $\om$-model $\mathsf{REC}$.
Thus, $\AMT$ is especially weak even among principles lying below
$\RT^2_2$.  It does, however, imply part (2) of the following theorem,
and therefore also $\OPT$ (\cite[Theorem 5.6~(2) and Corollary
5.8]{HSS-2009}).

\begin{thm}[Hirschfeldt, Shore, and Slaman {\cite[Theorem~5.7]{HSS-2009}}]\label{T:HSS_OPT_char} 
The following are equivalent over $\RCA$:
\begin{enumerate}
\item $\OPT$;
\item for every set $X$, there exists a set of degree hyperimmune
relative to~$X$.
\end{enumerate}
\end{thm}

This characterization was used by Hirschfeldt, Shore and
Slaman~\cite[p.~5831]{HSS-2009} to conclude that $\WKL$ does not imply
$\OPT$.  It is of interest to us in light of
Theorem~\ref{T:FIP_to_OPT} above, which links $F\IP$ with hyperimmune
degrees. Specifically, by Remark \ref{R:opt_thm_formal}, we have the following. 

\begin{cor} 
$\overline{D}_2\IP$ implies $\mathsf{OPT}$ over $\RCA$.
\end{cor}

The next proposition and theorem provide a partial step towards the converse of this
corollary.

\begin{prop}
Let $A = \langle A_i : i \in \N \rangle$
be a computable nontrivial family of sets.
Every set $D$ of degree hyperimmune relative to $\0'$ computes
a maximal subfamily of $A$ with the $F$ intersection property.
\end{prop}

\begin{proof}
We may assume that $A$ has no
finite maximal subfamily with the $F$ intersection property.  And by
deleting some of the members of $A$ if necessary, we may further
assume that $A_0 \neq \emp$.  Define a $\emp'$-computable function $g \colon
\N \to \N$ by letting $g(s)$ be the least $y$ such that for all finite
sets $F \subseteq \{0,\ldots,s\}$,

\[
\bigcap_{j \in F} A_j \neq \emp \Rightarrow (\exists x \leq y)[x \in
\bigcap_{j \in F} A_j].
\]
Since $D$ has hyperimmune degree relative to $\0'$, we may fix a
function $f \leq_T D$ not dominated by any $\emp'$-computable
function.  In particular, $f$ is not dominated by~$g$.

Now define $J \in \om^\om$ as follows.  Let $J(0) = 0$, and suppose
inductively that we have defined $J(s)$ for some $s \geq 0$.  Search
for the least $i \leq s$ not yet in the range of $J$ for which there
exists an $x \leq f(s)$ with
\[
x \in A_i \cap \bigcap_{j \leq s} A_{J(j)}.
\]
If it exists, set $J(s+1) = i$, and otherwise, set $J(s+1) = 0$

Clearly, $J \leq_T f$.  Moreover, $\bigcap_{i \leq s} A_{J(i)} \neq
\emp$ for every $s$, so the subfamily defined by $J$ has the $F$
intersection property.  We claim that for all $i$, if $A_i \cap
\bigcap_{j \leq s} A_{J(j)} \neq \emp$ for every $s$ then $i$ is in
the range of~$J$.  Suppose not, and let $i$ be the least witness to
this fact.  Since $f$ is not dominated by $g$, there exists an $s \geq
i$ such that $f(s) \geq g(s)$ and for all $t \geq s$, $J(t) \neq j$
for any $j < i$.  By construction, $J(j) \leq j$ for all $j$, so the
set $F = \{i\} \cup \{J(j) : j \leq s\}$ is contained in
$\{0,\ldots,s\}$.  Consequently, there necessarily exists some $x \leq
g(s)$ with $x \in A_i \cap \bigcap_{j \leq s} A_{J(j)}$.  But then $x
\leq f(s)$, so $J(s+1)$ is defined to be $i$, which is a
contradiction.  We conclude that $\langle A_{J(i)} : i \in \om
\rangle$ is maximal, as desired.
\end{proof}

\begin{thm}\label{T:permit} 
Let $A = \langle A_i : i \in \N \rangle$
be a computable nontrivial family of sets.
Every noncomputable computably enumerable set $W$ computes
a maximal subfamily of $A$ with the $F$ intersection property.
\end{thm}

\begin{proof}

As above, assume that $A$ has no
finite maximal subfamily with the $F$ intersection property, and that
 $A_0 \neq \emp$.  Fix a computable enumeration of $W$,
denoting by $W[s]$ the set of elements enumerated into $W$ by the end
of stage~$s$.  We construct a limit computable set $M$ by the method
of permitting, denoting by $M[s]$ the approximation to it at stage $s$ of the
construction.  For each $i$ and each $n$, call $\langle i,n \rangle$ a
\emph{copy} of~$i$.

\medskip
\noindent {\em Construction.}

\medskip
\noindent {\em Stage~$0$.}  Enumerate $\langle 0, 0 \rangle$ into
$M[0]$.

\medskip
\noindent {\em Stage $s+1$.}  Assume that $M[s]$ has been defined,
that it is finite and contains $\langle 0, 0 \rangle$, and that each
$i$ has at most one copy in $M[s]$.  For each $i$ with no copy in
$M[s]$, let $\ell(i,s)$ be the greatest $k$ with a copy in $M[s]$, if
it exists, such that there is an $x \leq s$ that belongs to $A_i$ and
to $A_j$ for every $j \leq k$ with a copy in $M[s]$.

Now consider all $i \leq s$ such that
\begin{itemize}
\item $\ell(i,s)$ is defined;

\item there is no $j$ with a copy in $M[s]$ such that $\ell(i,s) < j <
i$;

\item for each $\langle j,n \rangle \in M[s]$, if $\ell(i,s) < j$ then
$W [s] \res \langle j,n \rangle \neq W[s+1] \res \langle j,n \rangle$.

\end{itemize}
If there is no such $i$, let $M[s+1] = M[s]$.
Otherwise, fix the least such $i$, and let $M[s+1]$ be the result of
removing from $M[s]$ all $\langle j,n \rangle > \ell(i,s)$, and then
enumerating into it the least copy of $i$ greater than every element
of $M[s]$ and $W[s+1] - W[s]$.

\medskip
\noindent {\em End construction.}

\medskip For every $m$, if $M[s](m) \neq M[s+1](m)$ then $W[s] \res m
\neq W[s+1] \res m$.  Therefore, $M(m) = \lim_s M[s](m)$ exists for
all $m$ and is computable from~$W$.  Furthermore, note that
$\bigcap_{\langle i,n \rangle \in M[s]} A_i \neq \emp$ for all~$s$.
Thus, if $F$ is any finite subset of $M$, then $\bigcap_{\langle i, n
\rangle \in F} A_i \neq \emp$ since $F$ is necessarily a subset of
$M[s]$ for some~$s$.  If we now let $J : \om \to \om$ be any
$W$-computable function with range equal to $\{ i : (\exists
n)[\langle i,n \rangle \in M]\}$, it follows that $\langle A_{J(i)} :
i \in \om \rangle$ has the $F$ intersection property.

We claim that this subfamily is also maximal.  Seeking a
contradiction, suppose not, and let $i$ be the least witness to this
fact.  So $A_i \cap \bigcap_{\langle j,n\rangle \in F} A_j \neq \emp$
for every finite subset $F$ of $M$, and no copy of $i$ belongs to~$M$.
By construction, $\langle 0,0\rangle \in M[s]$ for all $s$ and hence
also to $M$, so it must be that $i > 0$.  Let $i_0, \ldots, i_r$ be
the numbers less than $i$ that have copies in $M$, and let these copies be
$\langle i_0,n_0 \rangle, \ldots, \langle i_r, n_r \rangle$,
respectively.  Let $s$ be large enough so that
\begin{itemize}
\item there is an $x \leq s$ with $x \in A_i \cap \bigcap_{j \leq n}
A_{i_j}$;

\item for all $t \geq s$ and all $j \leq n$, $\langle i_j, n_j \rangle
\in M[t]$.

\end{itemize} 
Now for all $t \geq s$, $\ell(i,t)$ is defined, and its
value must tend to infinity.

Note that no copy of $i$ can be in $M[t]$ at any stage $t \geq s$.
Otherwise, it would have to be removed at some later stage, which
could only be done for the sake of enumerating a copy of some number
$< i$.  This, in turn, could not be a copy of any of $i_0,\ldots,i_r$
by choice of $s$, and so it too would subsequently have to be removed.
Continuing in this way would result in an infinite regress, which is
impossible.

It follows that for each $t \geq s$ there is some $j > \ell(i,t)$ with
a copy $\langle j,n \rangle$ in $M[t]$.  Let $\langle j_t,n_t \rangle$
be the least such copy at stage~$t$.  Then $\langle j_t,n_t \rangle
\leq \langle j_{t+1},n_{t+1} \rangle$ for all $t$, since no $m <
\langle j_t,n_t \rangle$ can be put into $M[t+1]$.  Furthermore, for
infinitely many $t$ this inequality must be strict, since infinitely
often $\ell(i,t+1) \geq j_t$.

Now fix any $t \geq s$ so that $\ell(i,u) \geq i$ for all $u \geq t$.
Then for all $u \geq t$, $W [u] \res \langle j_t,n_t \rangle$ must be
equal to $W[u+1] \res \langle j_t,n_t \rangle$.  If not, we would
necessarily have $W [u] \res \langle j_u,n_u \rangle \neq W[u+1] \res
\langle j_u,n_u \rangle$, and hence $W [u] \res \langle j,n \rangle
\neq W[u+1] \res \langle j,n \rangle$ for every $\langle j,n \rangle
\in M[u]$ with $j > \ell(i,u)$.  But then some copy of $i$ would be
enumerated into $M[u+1]$, which cannot happen.  We conclude that for
all $u \geq t$, $W [u] \res \langle j_u,n_u \rangle = W \res \langle
j_u,n_u \rangle$.  Thus, given any $n$, we can compute $W \res n$
simply by searching for a $u \geq t$ with $\langle j_u, n_u \rangle
\geq x$.  This contradicts the assumption that $W$ is noncomputable.
The proof is complete.
\end{proof}

The above is of special
interest.  Heuristically, one would expect to be able to adapt a
permitting argument into one showing the same result but with ``every
noncomputable computably enumerable set'' replaced by ``every
hyperimmune set''.  For example, the proof in \cite{HSS-2009} that
$\OPT$ is implied over $\RCA$ by the existence of a set of hyperimmune
degree is an adaptation of a permitting argument of Csima~\cite[Theorem 1.2]{Csima-2004}.
The basic idea is to translate
receiving permissions into escaping domination by computable
functions.  We take a given function $f$ not dominated by any
computable one, and for each $i$ define a a computable function $g_i$
so that receiving permission for the $i$th requirement in the
permitting argument (such as putting a copy of $i$ into $M$)
corresponds to having $f(s) \geq g_i(s)$ for some~$s$.  But if we try
to do this in the case of Theorem~\ref{T:permit}, we run into the
problem of seemingly needing to know $f$ in order to define~$g$.
Intuitively, we are trying to put $A_i$ into our subfamily at stage $s$, and are
letting $g_i(s)$ be so large that it bounds a witness to the
intersection of $A_i$ and all the members of $A$ put in so far.  Thus,
the definition of $g_{i}(s)$ depends on which $A_j$ have been put in
at a stage $t < s$, i.e., on which $j$ had $f(t) > g_j(t)$ for some $t
< s$.  In the permitting argument this information is computable,
but here it is not.  We do not know of a way of get past
this difficulty, and thus leave open the question of whether $\OPT$
reverses to $F\IP$ (or $\overline{D}_2\IP$) over $\RCA$.

We also do not know whether the weaker implication from $\AMT$ to
$F\IP$ is provable in $\RCA$. However, the next proposition shows that
it is provable in $\RCA$ together with additional induction axioms. 
In particular, every $\om$-model of $\AMT$ is also a model of $F\IP$.
Thus we have a firm connection between the model-theoretic principles 
$\AMT$ and $\OPT$ and the set-theoretic principles $F\IP$ and $\overline{D}_n\IP$.

\begin{defn}[{Hirschfeldt and Shore~\cite[p.~5823]{HS-2007}}]\label{D:Pi01G} 
The following principle is defined in $\RCA$.
 
\begin{list}{\labelitemi}{\leftmargin=0em}\itemsep2pt
\item[]($\Pi^0_1\mathsf{G}$) For any uniformly $\Pi^0_1$ collection of
sets $D_i$, each of which is dense in $2^{<\N}$, there exists a set $G$
such that for every $i$, $G \res s \in D_i$ for some~$s$.
\end{list}
\end{defn}

\noindent Hirschfeldt, Shore and Slaman~\cite[Theorem~4.3, Corollary~4.5, and p.~5826]{HSS-2009} proved that $\Pi^0_1\mathsf{G}$ strictly
implies $\AMT$ over $\RCA$, and that $\AMT$ implies
$\Pi^0_1\mathsf{G}$ over $\RCA + \mathsf{I}\Sigma^0_2$.  As discussed
in the previous subsection, $\RCAo + Pi^0_1G$ is conservative over $\RCAo$
for restricted $\Pi^1_2$ sentences, and thus it does not imply
$\WKL$ over $\RCAo$.

\begin{prop}\label{P:Gen_to_FIP} 
$\Pi^0_1\mathsf{G}$ implies $F\IP$ over $\RCA$.
\end{prop}

\begin{proof} 
We argue in $\RCA$.  Let a nontrivial family $A =
\langle A_i : i \in \N \rangle$ be given.  We may assume $A$ has no
finite maximal subfamily with the $F$ intersection property.  Fix a
bijection $c : \N \to \N^{<\N}$.  Given $\sigma \in 2^{<\N}$, we say
that a number $x < |\sigma|$ is \textit{good for} $\sigma$ if
\begin{itemize}

\item $\sigma(x) = 1$;

\item $c(x) = \tau b$, which we call the {\em witness} of $x$, where

\begin{itemize}
  \item $\tau \in \N^{<\N}$,

  \item $b \in \N$,

  \item and there is a $y \leq b$ with $y \in \bigcap_{i < |\tau|}
  A_{\tau(i)}$.

\end{itemize}
\end{itemize} 
We define the \emph{{good} sequence} of $\sigma$ to be
either the empty string if there is no good number for $\sigma$, or
else the longest sequence $x_0 \cdots x_n \in \N^{<\N}$, $n \geq 0$,
where
\begin{itemize}
\item $x_0$ is the least good number for $\sigma$;

\item each $x_i$ is good, say with witness $\tau_i b_i$;

\item for each $i < n$, $x_{i+1}$ is the least good $x > x_i$ such
that if $\tau b$ is its witness then $\tau \succ \tau_i$.

\end{itemize} 
Note that $\Sigma^0_0$ comprehension suffices to prove
the existence of a function $2^{<\N} \to \N^{<\N}$ which assigns to
each $\sigma \in 2^{<\N}$ its good sequence.

Now for each $i \in \N$, let $D_i$ be the set of all $\sigma \in
2^{<\N}$ that have a nonempty good sequence $x_0\cdots x_n$, and if
$\tau b$ is the witness of $x_n$ then
\begin{itemize}
\item either $\tau(j) = i$ for some $j < |\tau|$,

\item or $A_i \cap \bigcap_{j < |\tau|} A_{\tau(j)} = \emp$.

\end{itemize} 
The $D_i$ are clearly uniformly $\Pi^0_1$, and it is not
difficult to see that they are dense in $2^{<\N}$.  Indeed, let
$\sigma \in 2^{<\N}$ be given, and define $b$, $j$, and $x$ as
follows.  If the good sequence of $\sigma$ is empty, choose the least
$j \geq i$ such that $A_j \neq \emp$ and let $b \geq \min A_j$ be
large enough that $x = c^{-1}(jb) \geq |\sigma|$.  If the good
sequence of $\sigma$ is some nonempty string $x_0 \cdots x_n$ and
$\tau b_n$ is the witness of $x_n$, choose the least $j \geq i$ such
that $A_j \cap \bigcap_{k < |\tau|} A_{\tau(k)} \neq \emp$ and let $b
\geq \min A_j \cap \bigcap_{k < |\tau|} A_{\tau(k)}$ be large enough
that $x = c^{-1}(\tau jb) \geq |\sigma|$.  In either case, $j$ exists
because of our assumption that $A$ is nontrivial and has no finite
maximal subfamily with the $F$ intersection property.  Now define
$\widetilde{\sigma} \in 2^{<\N}$ of length $x+1$ by
\[
\widetilde{\sigma}(y) =
\begin{cases} 
\sigma(y) & \text{if } y < |\sigma|,\\ 
0 & \text{if }|\sigma| \leq y < x,\\
1 & \text{if } y = x
\end{cases}
\]
to get an extension of $\sigma$ that belongs to $D_i$.

Apply $\Pi^0_1\mathsf{G}$ to the $D_i$ to obtain a set $G$ such that
for all $i$, there is an $s$ with $G \res s \in D_i$.  Note, that by
definition, each such $s$ must be nonzero, and $G \res s$ must have a
nonempty good sequence.  Notice that if $s \leq t$ then the good
sequence of $G \res t$ extends (not necessarily properly) the good
sequence of $G \res s$.  Furthermore, our assumption that $A$ has no
finite maximal subfamily with the $F$ intersection property implies
that the good sequences of the initial segments of $G$ are arbitrarily
long.

Now find the least $s$ such that $G \res s$ has a nonempty good
sequence, and for each $t \geq s$, if $x_0\cdots x_n$ is the good
sequence of $G \res t$, let $\tau_t b_t$ be the witness of $x_n$.  By
the preceding paragraph, we have $\tau_t \preceq \tau_{t+1}$ for all
$t$, and $\lim_t |\tau_t| = \infty$.  Let $J = \bigcup_{t \geq s}
\tau_t$, which exists by $\Sigma^0_0$ comprehension.  It is
straightforward to check that $B = \langle A_{J(i)}: i \in \N \rangle$
is a maximal subfamily of $A$ with the $F$ intersection property.
\end{proof}

We end this section with the result that $F\IP$ does not imply
$\Pi^0_1\mathsf{G}$ or even $\AMT$.  Csima, Hirschfeldt, Knight, and
Soare~\cite[Theorem~1.5]{CHKS-2004} showed that for every set ${D}
\leq_T \emp'$, if every complete atomic decidable theory has an atomic
model computable in $D$, then $D$ is non-low$_2$.  Thus $\AMT$
cannot hold in any $\om$-model all of whose sets have low$_2$ degree.
In conjunction with Theorem~\ref{T:permit}~(2), this fact allows us to
separate $F\IP$ and $\AMT$.

\begin{figure}
\[
{\xymatrix@R-15pt@C-10pt{ \ACA \ar@2[dddd] \ar@2[dr] \ar@1[r] & D_n\IP
\ar@1[l] \\ & \Pi^0_1\mathsf{G}
\ar@/_0pc/[dddl] |-{\object@{|}} |>{\object@{}} \ar@2[ddd] \ar@2[dr]
\\ & & F\IP \ar@/_0pc/[ddl] |-{\object@{|}} |>{\object@{}} \ar@1[d]
 \\ & & \vdots
\ar@1[d] \\ \WKL \ar@/_0pc/[dddr] |-{\object@{|}} |>{\object@{}}
\ar@2[dddd] & \AMT \ar@2[ddd] & \overline{D}_n \IP \ar@1[d] \\ & &
\vdots \ar@1[d] \\ & & \overline{D}_2 \IP \ar@1[dl] \\ & \OPT \ar@2[dl] \\ \RCA }}
\]
\caption{A summary of the results of Section~\ref{S:FIP}, with
\mbox{$n \geq 2$} being arbitrary.  Arrows denote implications
provable in $\RCA$, double arrows denote implications which are known
to be strict, and negated arrows indicate nonimplications.}
\end{figure}
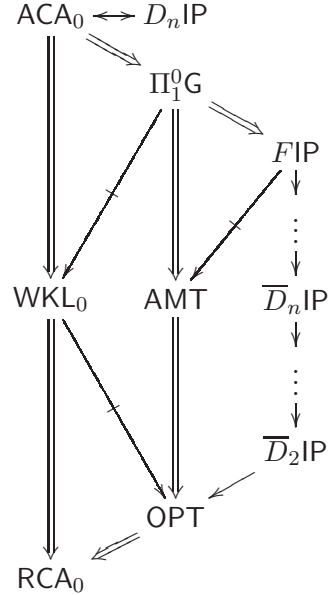

\begin{cor} 
For every noncomputable computably enumerable set $W$,
there exists an $\om$-model $\M$ of $\RCA + F\IP$ with $X \leq_T W$
for all $X \in \M$.  Therefore $F\IP$ does not imply
$\mathsf{AMT}$  over $\RCA$.
\end{cor}

\begin{proof} 
By Sacks's density theorem, there exist computably
enumerable sets $\emp <_T W_0 <_T W_1 <_T \cdots < W$.  Let $\M$ be
the $\om$-model whose second-order part consists of all sets $X$ such
that $X \leq_T W_i$ for some~$i$.  For each $i$, Theorem
\ref{T:permit}~(2) relativized to $W_i$ implies that every
$W_i$-computable nontrivial family of sets has a $W_{i+1}$-computable
maximal subfamily with the $F$ intersection property.  Thus, $\M
\models F\IP$.  The second part follows by building $\M$ with $W$
low$_2$.
\end{proof}

\section{Properties of finite character}\label{S:FCP}

The last family of choice principles we study makes use of properties
of finite character, sometimes in conjunction with finitary closure
operators (see Definitions~\ref{D:CE} and~\ref{D:NCE}). We will show
that these principles are equivalent to well known subsystems of
arithmetic, unlike the intersection principles of the last section. 

\begin{defn}
A formula $\varphi$ with one free set variable $X$ is
said to be of {\em finite character} (or have the {\em finite
character property}) if $\varphi(\emp)$ holds and, for every set $A$,
$\varphi(A)$ holds if and only if $\varphi(F)$ holds for every finite
$F \subseteq A$.
\end{defn}

The following basic facts are provable in $\mathsf{ZF}$.

\begin{prop}\label{p:fcmonotone}
Let $\varphi(X)$ be a formula of finite character.
\begin{enumerate}
\item If $A \subseteq B$ and $\varphi(B)$ holds then $\varphi(A)$
holds.
\item If $A_0 \subseteq A_1 \subseteq A_2 \subseteq \cdots$ is a
sequence of sets such that $\varphi(A_i)$ holds for each $i\in\om$,
then $\varphi(\bigcup_{i \in \om} A_i)$ holds.
\end{enumerate}
\end{prop}

We restrict our attention to formulas of second-order arithmetic, and
consider countable analogues of several variants of the principle
asserting that for every formula of finite character, every set has a
maximal subset (under inclusion) satisfying that formula.  Since the
empty set satisfies any formula of finite character by definition, the
validity of this principle can be seen by a simple application of
Zorn's lemma.

The formalism here will be simpler than that in the previous section
because we are dealing only with sets and their subsets, rather than
with families of sets and their subfamilies.  All the intersection
properties studied in Section~\ref{S:FIP} can, in principle, be thought of as being
defined by formulas of finite character.  For example, given a family
$A = \langle A_i : i \in \N \rangle$, the formula $(\forall i)(\forall
j)[A_i \cap A_j \neq \emp]$ has the finite character property, and if
$J = \{j_0 < j_1 < \cdots\}$ is a maximal subset of $\N$ satisfying
it, then $\langle A_{j_i} : i \in \N \rangle$ is a maximal subfamily
of $A$ with the $\overline{D}_2$ intersection property.  However, such
an analysis of $\overline{D}_2\IP$ would be too crude in light of
Proposition~\ref{P:familydef}.  Therefore, our focus in this section
will instead be on formulas of finite character in general, and on 
the strengths of principles based on formulas of finite character from 
restricted syntactic classes.

\subsection{\texorpdfstring{The scheme $\FCP$}{The scheme FCP}}

We begin with various forms of the following principle.

\begin{defn} 
The following scheme is defined in $\RCA$.
\begin{list}{\labelitemi}{\leftmargin=0em}\itemsep2pt
\item[]($\FCP$) For each formula $\varphi$ of finite character, which
may have arbitrary parameters, every set $A$ has a
\mbox{$\subseteq$-maximal} subset $B$ such that $\varphi(B)$ holds.
\end{list}
\end{defn}

\noindent In set theory, $\FCP$ corresponds to the principle $\mathsf{M} \, 7$ in
the catalog of Rubin and Rubin~\cite{RR-1985}, and is equivalent to
the axiom of choice~\cite[p.~34 and Theorem~4.3]{RR-1985}.

In order to better gauge the reverse mathematical strength of $\FCP$,
we consider restrictions of the formulas to which it applies.  As with
other such ramifications, we will primarily be interested in
restrictions to the classes in the arithmetical and analytical
hierarchies.  In particular, for each $i \in \{0,1\}$ and $n \geq 0$,
we make the following definitions:
\begin{itemize}
\item $\Sigma^i_n\text{-}\FCP$ is the restriction of $\FCP$ to
$\Sigma^i_n$ formulas;

\item $\Pi^i_n\text{-}\FCP$ is the restriction of $\FCP$ to $\Pi^i_n$
formulas;

\item $\Delta^i_n\text{-}\FCP$ is the scheme which says that for every
$\Sigma^i_n$ formula $\varphi(X)$ and every $\Pi^i_n$ formula
$\psi(X)$, if $\varphi(X)$ is of finite character and
\[ 
(\forall X)[\varphi(X) \leftrightarrow \psi(X)],
\]
then every set $A$ has a $\subseteq$-maximal set $B$ such that
$\varphi(B)$ holds.

\end{itemize} 
We also define $\QF\text{-}\FCP$ to be the restriction of $\FCP$ to the
class of quantifer-free formulas without parameters.

Our first result in this section is the following theorem, which will
allow us to neatly characterize most of the above restrictions of
$\FCP$ (see Corollary~\ref{c:fcpstrength}).  We draw attention to
part~(2) of the theorem, where $\Sigma^0_1$ does not appear in the list of
classes of formulas.  The reason behind this will be made apparent by
Proposition~\ref{P:Sig1_RCA}.

\begin{thm}\label{thm_main_fcp} 
For $i \in \{0,1\}$ and $n \geq 1 $,
let $\Gamma$ be any of $\Pi^i_n$, $\Sigma^i_n$, or~$\Delta^i_n$.
\begin{enumerate}
\item $\Gamma$-$\FCP$ is provable in $\Gamma$-$\CA$;

\item If\/ $\Gamma$ is $\Pi^0_n$, $\Pi^1_n$, $\Sigma^1_n$, or
$\Delta^1_n$, then $\Gamma$-$\FCP$ implies $\Gamma$-$\CA$
over~$\RCAo$.

\end{enumerate}
\end{thm}

We will make use of the following technical lemma in the proof (as
well as in the proof of Theorem~\ref{T:mainchar_CE} below). It is
needed only because there are no term-forming operations for sets in
$\Lang_2$. For example, there is no term in $\Lang_2$ that takes a set
$X$ and a canonical index $n$ and returns $X \cup D_n$. (Recall that
each finite (possibly empty) set of natural numbers is coded by a
unique natural number known as its \textit{canonical index}, and that
$D_n$ denotes the finite set with canonical index $n$.) The moral of
the lemma is that such terms can be interpreted into $\Lang_2$ in a
natural way.

The coding of finite sets by their canonical indices can be formalized
in $\RCAo$ in such a way that the predicate $i \in D_n$ is defined by
a formula $\rho(i,n)$ with only bounded quantifiers, and such that the
set of canonical indices is also definable by a bounded-quantifier
formula~\cite[Theorem II.2.5]{Simpson-2009}. Moreover, $\RCAo$ proves
that every finite set has a canonical index. We use the notation $Y =
D_n$ to abbreviate the formula $(\forall i)[ i \in Y \leftrightarrow
\rho(i,n)]$, along with similar notation for subsets of finite sets.

\begin{lem}\label{l:finiteset} 
Let $\varphi(X)$ be a formula with one
free set variable. There is a formula $\chat{\varphi}(x)$ with one
free number variable such that $\RCA$ proves
\begin{equation}\label{e:finiteset} 
(\forall A)(\forall n)[ A = D_n
  \to (\varphi(A) \leftrightarrow \chat{\varphi}(n))].
\end{equation}
Moreover, we may take $\chat{\varphi}$ to have the same
complexities in the arithmetical and analytic hierarchies
as~$\varphi$.
\end{lem}

\begin{proof} 
Let $\rho(i,n)$ be the formula defining the relation $i
\in D_n$, as discussed above. We may assume $\varphi$ is written in
prenex normal form. Form $\chat{\varphi}(n)$ by replacing each
occurrence $t \in X$ of $\varphi$, $t$ a term, with the formula
$\rho(t,n)$.

Let $\psi(X, \bar{Y}, \bar{m})$ be the quantifier-free matrix of
$\varphi$, where $\bar{Y}$ and $\bar{m}$ are sequences of variables
that are quantified in~$\varphi$. Similarly, let
$\chat{\psi}(n,\bar{Y},\bar{m})$ be the matrix of $\chat{\varphi}$.
Fix any model $\M$ of $\RCA$ and fix $n,A \in \M$ such that $\M
\models A = D_n$. A straightforward metainduction on the structure of
$\psi$ proves that
\[
\M \models (\forall \bar{Y})(\forall \bar{m})[ \psi(A, \bar{Y},
\bar{m}) \leftrightarrow \chat{\psi}(n, \bar{Y}, \bar{m}).
\] 

The key point is that the atomic formulas in $\psi(A, \bar{Y},
\bar{m})$ are the same as those in $\chat{\psi}(n,\bar{Y},\bar{m})$,
with the exception of formulas of the form $t \in A$, which have been
replaced with the equivalent formulas of the form~$\rho(t,n)$.

A second metainduction on the quantifier structure of $\varphi$ shows
that we may adjoin quantifiers to $\psi$ and $\chat{\psi}$ until we
have obtained $\varphi$ and $\chat{\varphi}$, while maintaining
logical equivalence.  Thus every model of $\RCA$
satisfies~(\ref{e:finiteset}).

Because $\rho$ has only bounded quantifiers, the substitution required
to pass from $\varphi$ to $\chat{\varphi}$ does not change the
complexity of the formula.
\end{proof}

If $F$ is any finite set and $n$ is its canonical index, we sometimes
write $\chat{\varphi}(F)$ for $\chat{\varphi}(n)$.

\begin{proof}[Proof of Theorem~\ref{thm_main_fcp}] 
For~(1), let $\varphi(X)$ and $A = \{a_i : i \in \N\}$ be an
instance of $\Gamma$-$\FCP$. Define $g \colon 2^{<\N} \times \N \to
2^{<\N}$ by
\[ 
g(\tau,i) =
\begin{cases} 
  1 & \text{if } \chat{\varphi}(\{ a_j : \tau(j) \convergesto 1\} 
       \cup \{a_i\}) \text{ holds},\\ 
0 & \text{otherwise}.
\end{cases}
\] 
where $\chat{\varphi}$ is as in the lemma, and for a finite set $F$,
$\chat{\varphi}(F)$ refers to $\chat{\varphi}(n)$ where $n$ is the
canonical index of~$F$.  The function $g$ exists by $\Gamma$
comprehension.  By primitive recursion, there exists a function $h
\colon \N \to 2^{<\N}$ such that for all $i \in \N$, $h(i) = 1$ if and
only if $g(h \res i, i) = 1$.  For each $i \in \N$, let $B_i = \{a_j :
j < i \land h(j) = 1\}$.  An induction on $\varphi$ shows that
$\varphi(B_i)$ holds for every $i\in \N$.

Let $B = \{ a_i : h(i) = 1\} = \bigcup_{i \in \N} B_i$.  Because
Proposition~\ref{p:fcmonotone} is provable in $\ACA$ and hence in
$\Gamma\text{-}\CA$, it follows that $\varphi(B)$ holds.  By the same
token, if $\varphi(B \cup \{a_k\})$ holds for some $k$ then so must
$\varphi(B_k \cup \{a_k\})$, and therefore $a_k \in B_{k+1}$, which
means that $a_k \in B$.  Therefore $B$ is $\subseteq$-maximal, and we
have shown that $\Gamma$-$\CA$ proves $\Gamma$-$\FCP$.

For~(2), we assume $\Gamma$ is one of $\Pi^0_n$, $\Pi^1_n$, or
$\Sigma^1_n$; the proof for $\Delta^1_n$ is similar. We work in $\RCA
+ \Gamma\text{-}\FCP$.  Let $\varphi(n)$ be a formula in $\Gamma$ and
let $\psi(X)$ be the formula $(\forall n)[n \in X \to \varphi(n)].$ It
is easily seen that $\psi$ is of finite character and belongs to
$\Gamma$.  By $\Gamma$-$\FCP$, $\N$ contains a $\subseteq$-maximal
subset $B$ such that $\psi(B)$ holds. For any~$y$, if $y \in B$ then
$\varphi(y)$ holds.  On the other hand, if $\varphi(y)$ holds then so
does $\psi(B \cup \{y\})$, so $y$ must belong to $B$ by maximality.
Therefore $B = \{y \in \N : \varphi(y) \}$, and we have shown that
$\Gamma$-$\FCP$ implies~$\Gamma$-$\CA$.
\end{proof}

The corollary below summarizes the theorem as it applies to the
various classes of formulas we are interested in.  Of special note
is part~(5), which says that $\FCP$ itself (that is, $\FCP$ for
arbitrary $\Lang_2$-formulas) is as strong as any theorem of
second-order arithmetic can be.

\begin{cor}\label{c:fcpstrength}
The following are provable in $\RCA$:
\begin{enumerate}
\item $\Delta^0_1$-$\FCP$, $\Sigma^0_0\text{-}\FCP$, and
$\QF\text{-}\FCP$;

\item for each $n \geq 1$, $\ACA$ is equivalent to $\Pi^0_n$-$\FCP$;

\item for each $n \geq 1$, $\Delta^1_n\text{-}\CA$ is equivalent to
$\Delta^1_n$-$\FCP$;

\item for each $n \geq 1$, $\Pi^1_n$-$\CA$ is equivalent to\/
$\Pi^1_n$-$\FCP$ and to $\Sigma^1_n$-$\FCP$;

\item $\mathsf{Z}_2$ is equivalent to $\FCP$.

\end{enumerate}
\end{cor}

The case of $\FCP$ for $\Sigma^0_1$ formulas is anomalous.  The proof
of part (2) of the theorem does not go through for $\Sigma^0_1$
because this class is not closed under universal quantification.  As
the proof of the next proposition shows, this limitation is quite
significant.  Intuitively, it means that a $\Sigma^0_1$ formula
$\varphi(X)$ of finite character can only control a fixed finite piece
of a set~$X$.  Hence, for the purposes of finding a maximal subset of
which $\varphi$ holds, we can replace $\varphi$ by a formula with only
bounded quantifiers.

\begin{prop}\label{P:Sig1_RCA} 
$\Sigma^0_1$-$\FCP$ is provable in $\RCA$.
\end{prop}

\begin{proof}
Let $\varphi(X)$ be a $\Sigma^0_1$ formula of finite
character.  We claim that there exists a finite subset $F$ of $\N$
such that for every set $A$, if $F \cap A = \emptyset$ then
$\varphi(A)$ holds.  Let $\psi(X,x)$ be a bounded quantifier formula
such that $\varphi(X) \equiv (\exists x)\, \psi(X,x)$, and fix $n$
such that $\psi(\emp,n)$ holds.  Note that $\psi(X,n)$ is a bounded
quantifier formula with no free number variables.  Any such formula is
equivalent to a quantifier-free formula, because each quantifier will
be bounded by a standard natural number. In turn, each quantifier-free
formula can be written as a disjunction of conjunctions of atomic
formulas and their negations.  So we may assume $\psi(X,n)$ is in this
form.  Since $\psi(\emp,n)$ holds, there must be a disjunct
$\theta(X)$ of $\psi(X,n)$ that holds of~$\emp$.  Clearly, $\theta(X)$
cannot have a conjunct of the form $t \in X$, $t$ a term.  Therefore,
if we let $F$ be the set of all terms $t$ such $t \notin X$ is a
conjunct of $\theta(X)$, we see that $\theta(A)$ holds whenever $F
\cap A = \emptyset$. This completes the proof of the claim.

Now fix any set~$A$.  By the claim, there is a finite set $F$ such
that $\varphi(A\setminus F)$ holds.  By $\Sigma^0_1$ induction, there
is such an $F$ of smallest size.  Then if $\varphi((A \setminus F)
\cup \{a\})$ holds for some $a \in A$, it cannot be that $a \in F$, as
otherwise $F' = F \setminus \{a\}$ would be a strictly smaller finite
set than $F$ such that $\varphi(A \setminus F')$ holds.  Thus it must
be that $a \in A \setminus F$, and we conclude that $A \setminus F$ is
a $\subseteq$-maximal subset of $A$ of which $\varphi$ holds.
\end{proof}

The above proof contains an implicit non-uniformity in the choice
of $F$ of smallest size.  The following proposition shows that this
non-uniformity is essential, by showing that a sequential form of
$\Sigma^0_1\text{-}\FCP$ is a strictly stronger principle.

\begin{prop}
The following are equivalent over $\RCA$:
\begin{enumerate}
\item $\ACA$;
\item for every family $A = \langle A_i : i \in \N \rangle$ of sets,
and every $\Sigma^0_1$ formula $\varphi(X,x)$ with one free set
variable and one free number variable such that for all $i \in \N$,
the formula $\varphi(X,i)$ is of finite character, there exists a
family $B = \langle B_i : i \in \N \rangle$ of sets such that for all
$i$, $B_i$ is a $\subseteq$-maximal subset of $A_i$ satisfying
$\varphi(X,i)$.
\end{enumerate}
\end{prop}

\begin{proof} The forward implication follows by a straightforward
modification of the proof of Theorem~\ref{thm_main_fcp}.  For the
reversal, let a one-to-one function $f \colon \N \to \N$ be given.
For each $i \in \N$, let $A_i = \{i\}$, and let $\varphi(X,x)$ be the
formula
\[ 
(\exists y)[x \in X \to f(y) = x].
\] 
Then, for each $i$, $\varphi(X,i)$ has the finite character
property, and for every set $S$ that contains $i$, $\varphi(S,i)$
holds if and only if $i \in \operatorname{range}(f)$.  Thus, if $B =
\langle B_i : i \in \N \rangle$ is the subfamily obtained by applying
part (2) to the family $A = \langle A_i : i \in \N \rangle$ and the
formula $\varphi(X,x)$, then
\[ 
i \in \operatorname{range}(f) \Leftrightarrow B_i = \{i\}
\Leftrightarrow i \in B_i.
\] 
It follows that the range of $f$ exists.
\end{proof}

\noindent Note that the proposition fails for the class of
bounded-quantifier formulas of finite character in place of the class
of $\Sigma^0_1$ such formulas, since part (2) is then clearly provable
in $\RCA$.  Thus, in spite of the similarity between the two classes
suggested by the proof of Proposition~\ref{P:Sig1_RCA}, the two do not
coincide.

\subsection{Finitary closure operators}\label{S:CE}

We can strengthen $\FCP$ by imposing additional requirements on the
maximal set being constructed. In particular, we now consider
requiring the maximal set to satisfy a finitary closure property as
well as to satisfy a property of finite character.

\begin{defn}\label{D:CE} 
A \textit{finitary closure operator} is a set
of pairs $\langle F, n \rangle$ in which $F$ is (the canonical index
for) a finite (possibly empty) subset of $\N$ and $n \in \N$. A set $A
\subseteq \N$ is \textit{closed} under a finitary closure operator
$D$, or \emph{$D$-closed}, if for every $\langle F, n \rangle \in D$,
if $F \subseteq A$ then $n \in A$.
\end{defn}

\noindent Our definition of a closure operator is not the standard set-theoretic
definition presented by Rubin and Rubin~\cite[Definition 6.3]{RR-1985}.
However, it is easy to see that for each operator of the one kind
there is an operator of the other such that the same sets are closed
under both.  The above definition has the advantage of being readily
formalizable in $\RCA$.

The following fact expresses the monotonicity of finitary closure
operators.

\begin{proposition}\label{p:clmonotone}
If $D$ is a finitary closure
operator and $A_0 \subseteq A_1 \subseteq A_2 \cdots$ is a sequence of
sets such that each $A_i$ is $D$-closed, then $\bigcup_{i \in \N} A_i$
is $D$-closed.
\end{proposition}

The principle in the next definition is analogous to principle
$\mathsf{AL}' \, 3$ of Rubin and Rubin~\cite{RR-1985}, which is
equivalent to the axiom of choice by~\cite[p.~96, and Theorems~6.4 and~6.5]{RR-1985}.

\newcommand{\AL}{\mathsf{CE}}

\begin{defn} 
The following scheme is defined in $\RCAo$.
\begin{list}{\labelitemi}{\leftmargin=0em}\itemsep2pt
\item[]($\AL$) If $D$ is a finitary closure operator, $\varphi$ is a
formula of finite character, and $A$ is any set, then every $D$-closed
subset of $A$ satisfying $\varphi$ is contained in a maximal such
subset.
\end{list}
\end{defn}

\noindent In the terminology of Rubin and Rubin~\cite{RR-1985}, this
is a ``primed'' statement, meaning that it asserts the existence not
merely of a maximal subset of a given set, but the existence of a
maximal \emph{extension} of any given subset.  Primed versions of all of
the principles considered above can be formed, and can easily be seen
to be equivalent to the unprimed ones.  By contrast, $\AL$ has only a
primed form.  This is because if $A$ is a set, $\varphi$ is a formula
of finite character, and $D$ is a finitary closure operator, $A$ need
not have any $D$-closed subset of which $\varphi$ holds.  For example,
suppose $\varphi$ holds only of $\emp$, and $D$ contains a pair of the
form $\langle \emp, a \rangle$ for some $a \in A$.

This leads to the observation that the requirements in the $\AL$
scheme that the maximal set must both be $D$-closed and satisfy a
property of finite character are, intuitively, in opposition to each
other.  Satisfying a finitary closure property is a positive
requirement, in the sense that forming the closure of a set usually
requires adding elements to the set. Satisfying a property of finite
character can be seen as a negative requirement in light of part~(1)
of Proposition~\ref{p:fcmonotone}.

We consider restrictions of $\AL$ as we did restrictions of $\FCP$
above.  By analogy, if $\Gamma$ is a class of formulas, we use the
notation $\Gamma\text{-}\AL$ to denote the restriction of $\AL$ to the
formulas in~$\Gamma$.  We begin with the following analogue of Theorem
\ref{thm_main_fcp}~(1) from the previous subsection.

\begin{thm}\label{T:mainchar_CE}
For $i \in \{0,1\}$ and $n \geq 1 $,
let $\Gamma$ be $\Pi^i_n$, $\Sigma^i_n$, or $\Delta^1_n$.
Then $\Gamma$-$\AL$ is provable in $\Gamma$-$\CA$.
\end{thm}

\begin{proof} 
We work in $\Gamma\text{-}\CA$. Let $\varphi$ be a formula of finite character
in $\Gamma$, which may have parameters, and let $D$ be a finitary closure
 operator. Let $A$ be any set and
let $C$ be a $D$-closed subset of $A$ such that $\varphi(C)$ holds.
  
For any $X \subseteq A$, let $\cl_D(X)$ denote the \emph{$D$-closure}
of $X$.  That is, $\cl_D(X) = \bigcup_{i \in \N} X_i$, where
$X_0 = X$ and for each $i \in \N$, $X_{i+1}$ is the set of all $n \in
\N$ such that either $n \in X_i$ or there is a finite set $F \subseteq
X_i$ such that $\langle F,n \rangle \in D$.  Because we take
$D$ to be a set, $\cl_D(X)$ can be defined using a $\Sigma^0_1$ formula with 
parameter $D$.
Define a formula $\psi(\sigma, X)$ by
\begin{align*}
\psi(\sigma, X)  \Leftrightarrow {} 
& (\forall n)[ ( D_n \subseteq \cl_D(X
\cup \{i : \sigma(i) = 1\}) ) \to \chat{\varphi}(n)] \\
&\wedge \cl_D(X \cup \{i : \sigma(i) = 1\}) \subseteq A,
\end{align*}
where $\chat{\varphi}$ is as in Lemma~\ref{l:finiteset}. Note that 
$\psi$ is
arithmetical if $\Gamma$ is $\Pi^0_n$ or $\Sigma^0_n$, and is in
$\Gamma$ otherwise.

Define the function $f \colon \N \to \{0,1\}$ inductively such that $f(i)
= 1$ if and only if $\psi(\{j < i : f(j) = 1\} \cup \{i\}, C)$ holds.
The characterization of the complexity of $\psi$ ensures that $f$ can be constructed using $\Gamma$ comprehension. Now let
\[ 
B_i = \cl_D(C \cup \{ j < i : f(j) = 1\})
\]
for each $i \in \N$, and let $B = \bigcup_{i \in \N} B_i$.  The
construction of $f$ ensures that $\varphi(B_i)$ implies
$\varphi(B_{i+1})$ for all~$i$, and we have assumed that $\varphi$ holds of
$B_0 = \cl_D(C) = C$. Therefore, an instance of induction
shows that $\varphi$ holds of $B_i$ for all $i \in \N$, and thus also
of $B$ by Proposition~\ref{p:fcmonotone}.  This also shows that $B
\subseteq A$.  Similarly, because each $B_i$ is $D$-closed, the
formalized version of Proposition~\ref{p:clmonotone} implies $B$ is
$D$-closed.

Finally, we check that $B$ is a maximal $D$-closed extension of $C$ in
$A$ of which $\varphi$ holds. Suppose that for some $i \in A$, $B \cup
\{i\}$ is $D$-closed and $\varphi(B \cup \{i\})$ holds.  Then since
$B_i \subseteq B$, we have $\cl_D(B_i \cup \{i\}) \subseteq B \cup
\{i\}$.  Thus $\varphi(F)$ holds for every finite subset $F$ of
$\cl_D(B_i \cup \{i\})$, so by definition $f(i) = 1$ and $B_{i+1} =
\cl_D(B_i \cup \{i\})$.  Here we are using the fact that for all sets
$X$ and all $a,b \in \N$, $\cl_D(X \cup \{a,b\}) = \cl_D(\cl_D(X \cup
\{a\}) \cup \{b\})$.  Since $B_{i+1} \subseteq B$, we conclude that $i
\in B$, as desired.
\end{proof}

It follows that for most standard classes $\Gamma$,
$\Gamma\text{-}\AL$ is equivalent to $\Gamma\text{-}\FCP$.  Indeed,
for any class $\Gamma$ we have that $\Gamma\text{-}\AL$ implies
$\Gamma\text{-}\FCP$, because any instance of the latter can be regarded
as an instance of the former by adding an empty finitary closure
operator.  And if $\Gamma$ is $\Pi^0_n$, $\Pi^1_n$, $\Sigma^1_n$, or
$\Delta^1_n$, then $\Gamma\text{-}\FCP$ is equivalent to
$\Gamma\text{-}\CA$ by Theorem~\ref{thm_main_fcp}~(2), and hence
reverses to $\Gamma\text{-}\AL$.  Thus, in particular, parts (2)--(5)
of Corollary~\ref{c:fcpstrength} hold for $\AL$ in place of $\FCP$,
and the full scheme $\AL$ itself is equivalent to $\mathsf{Z}_2$.

The proof of the preceding theorem does not work for $\Gamma =
\Delta^0_1$, because then $\Gamma\text{-}\CA$ is just $\RCA$, and we
need at least $\ACA$ to prove the existence of the function $f$
defined there (the formula $\psi(\sigma,X)$ being arithmetical at
best).  The next proposition shows that this cannot be avoided, even
for a class of considerably weaker formulas.

\begin{prop}\label{P:alqf_implies_aca} 
$\QF\text{-}\AL$ implies $\ACA$ over $\RCA$.
\end{prop}

\begin{proof}
Assume a one-to-one function $f : \N \to \N$ is given.
Let $\varphi(X)$ be the quantifier-free formula $0 \notin X$, which
trivially has finite character, and let \mbox{$\langle p_i: i \in \N
\rangle$} be an enumeration of all primes.  Let $D$ be the finitary
closure operator consisting, for all $i, n \in \N$, of all pairs of
the form
\begin{itemize}
\item $\langle \{p_i^{n+1}\},p_i^{n+2} \rangle$;

\item $\langle \{p_i^{n+2}\},p_i^{n+1} \rangle$;

\item $\langle \{p_i^{n+1}\},0 \rangle$, if $f(n) = i$.
\end{itemize} 
Notice that $D$ exists by $\Delta^0_1$ comprehension
relative to $f$ and our enumeration of primes.

Note that $\emp$ is a $D$-closed subset of $\N$ and $\varphi(\emp)$
holds.  Thus, we may apply $\AL$ for quantifier-free formulas to
obtain a maximal $D$-closed subset $B$ of $\N$ such that $\varphi(B)$
holds.  Then by definition of $D$, for every $i \in \N$, $B$ either
contains every positive power of $p_i$ or no positive power.  Now if
$f(n) = i$ for some $n$, then no positive power of $p$ can be in $B$,
since otherwise $p^{n+1}$ would necessarily be in $B$ and hence so
would~$0$.  On the other hand, if $f(n) \neq i$ for all $n$ then $B
\cup \{p_i^{n+1} : n \in \N \}$ is $D$-closed and satisfies $\varphi$,
so by maximality $p^{n+1}_i$ must belong to $B$ for every~$n$.  It
follows that $i \in \operatorname{range}(f)$ if and only if $p_i \in
B$, so the range of $f$ exists.
\end{proof}

Thus we are able to separate $\AL$ from $\FCP$ at least in terms of
some of their strictest restrictions.  In contrast to Corollary
\ref{c:fcpstrength}~(1) and Proposition~\ref{P:Sig1_RCA}, we
consequently have:

\begin{cor}\label{c:alequiv} 
The following are equivalent over $\RCA$:
\begin{enumerate}
\item $\ACA$;
\item $\Sigma^0_1\text{-}\AL$;
\item $\Sigma^0_0\text{-}\AL$;
\item $\QF\text{-}\AL$.
\end{enumerate}
\end{cor}

We conclude this subsection with one additional illustration of how
formulas of finite character can be used in conjunction with finitary
closure operators.  Recall the following concepts from order theory:
\begin{itemize}
\item A \textit{countable join-semilattice} is a countable poset
$\langle L, \leq_L \rangle $ with a maximal element $1_L$ and an
operation $\lor_L \colon L \times L \to L$ such that for all $a,b \in
L$, $a \lor_L b$, called the \emph{join} of $a$ and $b$, is the least
upper bound of $a$ and~$b$.

\item An \textit{ideal} on a countable join-semilattice $L$ is a
subset $I$ of $L$ that is downward closed under $\leq_L$ and closed
under $\lor_L$.

\end{itemize} 
The principle in the following proposition is the
countable analogue of a variant of $\mathsf{AL}' \, 1$ in Rubin and
Rubin~\cite{RR-1985}; compare with Proposition~\ref{P:NCE_ideals}
below.  For more on the computability theory of ideals on lattices,
see Turlington~\cite{Turlington-2010}.

\begin{prop}\label{p:alextend}
Over $\RCA$, $\QF\text{-}\AL$ implies that every proper
ideal on a countable join-semilattice extends to a maximal proper ideal.
\end{prop}

\begin{proof}
Let $L$ be a countable join-semilattice.  Let $\varphi$
be the formula $1 \not \in X$, and let $D$ be the finitary closure
operator consisting of all pairs of the form
\begin{itemize}
\item $\langle \{a,b\}, c\rangle$ where $a,b \in L$ and $c = a \lor b$;
\item $\langle \{a\}, b\rangle$, where $b \leq_L a$.
\end{itemize}
Because we define a join-semilattice to come with both
the order relation and the join operation, the set $D$ is $\Delta^0_0$
with parameters, so $\RCAo$ proves $D$ exists. It is immediate that a
set $X$ is closed under $D$ if and only if $X$ is an ideal in~$L$.
\end{proof}

\subsection{Nondeterministic finitary closure operators}\label{S:NCE}

It appears that the underlying reason that the restriction of $\AL$
to arithmetical formulas is provable in $\ACAo$ (and more generally,
why $\Gamma\text{-}\AL$ is provable in $\Gamma\text{-}\CA$ if $\Gamma$
is as in Theorem~\ref{T:mainchar_CE}) is that our definition of
finitary closure operator is very constraining.  Intuitively, if $D$
is such an operator and $\varphi$ is an arithmetical
formula, and we seek to extend some $D$-closed subset 
$B$ satisfying $\varphi$ to a maximal
such subset, we can focus largely on ensuring 
that $\varphi$ holds.  Achieving closure under $D$ is
relatively straightforward, because at each stage we only need to search through
all finite subsets $F$ of our current extension, and then adjoin all $n$
such that $\langle F,n \rangle \in D$.  This closure process becomes
far less trivial if we are given a choice of which elements
to add. We now consider the case when each finite subset $F$
can be associated with a possibly infinite set of numbers  
from which we must choose at least one to adjoin. We will show that this weaker
notion of closure operator leads to a stronger analogue of $\AL$.

\begin{defn}\label{D:NCE}
A \textit{nondeterministic finitary closure operator} is a sequence 
of sets of the form $\langle F, S\rangle$ where $F$ is (the canonical index for)
a finite (possibly empty) subset 
of $\N$ and $S$ is a nonempty subset of~$\N$. A set $A \subseteq \N$ 
is \textit{closed} under a nondeterministic finitary closure operator 
$N$, or $N$-closed, if for each $\langle F, S \rangle$ in $N$, if 
$F \subseteq A$ then $A \cap S \neq \emp$.
\end{defn}

Note that if $D$ is a \emph{deterministic} finitary closure operator,
that is, a finitary closure operator in the stronger sense of the
previous subsection, then for any set $A$ there is a unique
$\subseteq$-minimal $D$-closed set extending~$A$. This is not true for
nondeterministic finitary closure operators. Let $N$ be the 
operator such that $\langle\emptyset,\N\rangle \in N$ and, for each $i
\in \N$ and each $j > i$, $\langle\{i\},\{j\}\rangle \in N$.  Then any
$N$-closed set extending $\emptyset$ will be of the form $\{i \in \N :
i \geq k\}$ for some~$k$, and any set of this form is $N$-closed. Thus
there is no $\subseteq$-minimal $N$-closed set.

In this subsection we study the following nondeterministic version of
$\AL$.

\begin{defn}
The following scheme is defined in $\RCAo$.
\begin{list}{\labelitemi}{\leftmargin=0em}\itemsep2pt
\item[]($\mathsf{NCE}$) If $N$ is a nondeterministic closure operator,
$\varphi$ is a formula of finite character, and $A$ is any set, then
every $N$-closed subset of $A$ satisfying $\varphi$ is contained in a
maximal such subset.
\end{list}
\end{defn}

\noindent Restrictions of $\NCE$ to various syntactical classes of formulas are
defined as for $\AL$ and $\FCP$. Note that, because the union of a
chain of $N$-closed sets is again $N$-closed, $\NCE$ can be
proved in set theory using Zorn's lemma.

\begin{rem}\label{rem:nce_remark} 
We might expect to be able to prove
$\NCE$ from $\AL$ by suitably transforming a given nondeterministic
finitary closure operator $N$ into a deterministic one.  For instance,
we could go through the members of $N$ one by one, and
for each such member $\langle F,S \rangle$ add $\langle F, n \rangle$
to $D$ for some $n \in S$ (e.g., the least $n$).  All $D$-closed sets
would then indeed be $N$-closed.  The converse, however, would not
necessarily be true, because a set could have $F$ as a subset for some
$\langle F,S \rangle \in N$, yet it could contain a different $n \in
S$ than the one chosen in defining~$D$.  In particular, a maximal
$D$-closed subset (of some given set) would not need to be maximal
among $N$-closed subsets.
\end{rem}

The following result provides a simple but concrete example of this
point.  Recall that an \emph{ideal} on a countable poset $\langle P,
\leq_P \rangle$ is a subset $I$ of $P$ downward closed under $\leq_P$
and such that for all $p,q \in I$ there is an $r \in I$ with $p \leq_P
r$ and $q \leq_P r$.  The next proposition is similar to
Proposition~\ref{p:alextend} above, which dealt with ideals on
countable join-semilattices.  In the proof of that proposition, we
defined a deterministic finitary closure operator $D$ in such a way
that $D$-closed sets were closed under the join operation.  For this
we relied on the fact that for every two elements in the semilattice
there is a unique element that is their join.  The reason we need
nondeterministic finitary closure operators below is that, for ideals
on countable posets, there are no longer unique elements witnessing
closure under the relevant operations. 

\begin{prop}\label{P:NCE_ideals} 
Over $\RCAo$,
$\Pi^0_2\text{-}\mathsf{NCE}$ implies that every ideal on a countable
poset can be extended to a maximal ideal.
\end{prop}

\begin{proof} 
We work in $\RCAo$. Let $\langle P, \leq_P \rangle$ be a
countable poset. Without loss of generality we may assume $P = \{ p_i
: i \in \N\}$ is infinite.  We form a nondeterministic closure
operator $N= \langle N_i : i \in \N\rangle$ by considering the
following two cases. For each $i \in \N$,
\begin{itemize}
\item if $i = 2\langle j,k\rangle$ and $p_j \leq_P p_k$, let $N_i =
\langle \{p_k\},\{p_j\}\rangle$;

\item if $i = 2\langle j,k,l\rangle + 1$ and $p_j \leq_P p_l$ and $p_k
\leq_P p_l$, let
\[ 
N_i = \langle \{p_j,p_k\}, \{p_n : (p_j \leq_P p_n) 
   \land (p_k \leq_P p_n)\}\rangle;
\]

\item otherwise, let $N_i = \langle\{p_i\},\{p_i\}\rangle$.

\end{itemize}
This construction gives a quantifier-free definition of
each $N_i$ uniformly in $i$, so the sequence $N$ exists.

Let $\varphi(X)$ be the $\Pi^0_2$ formula which says that every pair
of elements in $X$ has a common upper bound in~$P$. A straightforward
proof shows that $\varphi$ is of finite character and that a set $I
\subseteq P$ is an ideal on $P$ if and only if $I$ is $N$-closed and
$\varphi(I)$ holds.
\end{proof}

Mummert~\cite[Theorem~2.4]{Mummert-2006} showed that the proposition
that every ideal on a countable poset extends to a maximal ideal is
equivalent to $\Pi^1_1\text{-}\CA$ over $\RCA$.  Hence,
$\Pi^0_2\text{-}\mathsf{NCE}$ implies $\Pi^1_1\text{-}\CA$.  By
Theorem~\ref{T:mainchar_CE}, $\Pi^0_2\text{-}\AL$ is provable in
$\ACA$, so we see that the idea of Remark~\ref{rem:nce_remark}
fundamentally cannot work.

We will obtain the reversal of $\Pi^0_2\text{-}\NCE$ to
$\Pi^1_1\text{-}\CA$ in a sharper form in Theorem~\ref{t:ncereverse}
below.  First, we prove the following upper bound.  The proof uses a
technique involving countable coded $\beta$-models, parallel to
Lemma~2.4 of Mummert~\cite{Mummert-2006}.  In $\RCA$, a
\emph{countable coded $\beta$-model} is defined as a sequence $\M =
\langle M_i : i \in \N \rangle$ of subsets of $\N$ such that for every $\Sigma^1_1$
formula $\varphi$ with parameters from $\M$, $\varphi$ holds if and
only if $\M \models \varphi$ \cite[Definitions VII.2.1 and
VII.2.3]{Simpson-2009}.  A general treatment of countable coded
$\beta$-models is given by Simpson~\cite[Section~VII.2]{Simpson-2009}.

\begin{prop}\label{p:nceprovable} 
$\Sigma^1_1\text{-}\mathsf{NCE}$ is provable in $\Pi^1_1\text{-}\CA$.
\end{prop}

\begin{proof}
We work in $\Pi^1_1\text{-}\CA$. Let $\varphi$ be a
$\Sigma^1_1$ formula of finite character (possibly with parameters)
and let $N$ be a nondeterministic closure operator. Let $A$ be any set
and let $C$ be an $N$-closed subset of $A$ such that $\varphi(C)$
holds.

Let $\M = \langle M_i : i \in \N\rangle$ be a countable coded
$\beta$-model containing $A$, $B$, $N$, and any parameters of
$\varphi$, which exists by \cite[Theorem
VII.2.10]{Simpson-2009}. Using $\Pi^1_1$ comprehension,
we may form the set $\{i : \M \models \varphi(M_i)\}$.

Working outside $\M$, we build an increasing sequence $\langle B_i :
i \in \N\rangle$ of $N$-closed extensions of~$C$. Let $B_0 = C$.
Given~$i$, ask whether there is a $j$ such that
\begin{itemize}
\item $M_j$ is an $N$-closed subset of $A$;
\item $B_i \subseteq M_j$;
\item $i \in M_j$;
\item and $\varphi(M_j)$ holds.
\end{itemize} 
If there is, choose the least such $j$ and let
$B_{i+1} = M_j$. Otherwise, let $B_{i+1} = B_i$.  Finally, let $B =
\bigcup_{i\in \N} B_i$.

Because the inductive construction only asks arithmetical questions
about $\M$, it can be carried out in $\Pi^1_1\text{-}\CA$, and so
$\Pi^1_1\text{-}\CA$ proves that $B$ exists. Clearly $C \subseteq B
\subseteq A$.  An arithmetical induction shows that for all $i \in
\N$, $\varphi(B_i)$ holds and $B_i$ is $N$-closed.  Therefore, the
formalized version of Proposition~\ref{p:fcmonotone} shows that
$\varphi(B)$ holds, and the analogue of Proposition~\ref{p:clmonotone}
to nondeterministic finitary closure operators shows that $B$ is
$N$-closed.

Now suppose that for some $i \in A$, $B \cup \{i\}$ is an $N$-closed
subset of $A$ extending $C$ and satisfying~$\varphi$.  Because $\varphi$
is $\Sigma^1_1$, and because $N$ is a sequence, the property
\begin{equation}\label{eq:betamod} 
(\exists X)[X \text{ is
$N$-closed} \land B_i \subseteq X \subseteq A \land i \in X \land
\varphi(X)]
\end{equation}
\noindent is expressible by a $\Sigma^1_1$ sentence, and $B \cup
\{i\}$ witnesses that it is true.  Thus, because $\M$ is a
$\beta$-model, this sentence must be satisfied by $\M$, which means that
some $M_j$ must also witness it. The inductive construction must
therefore have selected such an $M_j$ to be $B_{i+1}$, which means $i
\in B_{i+1}$ and hence $i \in B$.  It follows that $B$ is maximal.
\end{proof}

The next theorem shows that $\mathsf{NCE}$ for quantifier-free
formulas without parameters is already as strong as
$\Sigma^1_1\text{-}\FCP$ and $\Sigma^1_1\text{-}\AL$.  In particular,
in view of Corollary~\ref{c:alequiv}, it is considerably stronger than
$\AL$ for quantifier-free formulas.

\begin{thm}\label{t:ncereverse} 
For each $n \geq 1$, the following are equivalent over $\RCAo$:
\begin{enumerate}
\item $\Pi^1_1\text{-}\CA$;
\item $\Sigma^1_1\text{-}\mathsf{NCE}$;
\item $\Sigma^0_n\text{-}\mathsf{NCE}$;
\item $\QF\text{-}\mathsf{NCE}$.
\end{enumerate}
\end{thm}

\begin{proof} 
We have already proved (1) implies (2), and it is
obvious that (2) implies (3) and (3) implies (4). The reversal of (4)
to (1) splits into two steps.

For the first step, note that $\RCAo$ can convert any finitary closure
operator $D$ into a corresponding nondeterministic closure operator
$N$ such that the notions of $D$-closed and $N$-closed coincide (note
that this is the opposite of what was discussed in Remark
\ref{rem:nce_remark}).  Therefore $\mathsf{NCE}$ for quantifier-free
formulas implies $\ACAo$ over~$\RCAo$ by
Proposition~\ref{P:alqf_implies_aca}.

Next, for the second step, we work in $\ACAo$.  Let $\langle T_i : i \in
\N\rangle$ be a sequence of subtrees of $\N^{<\N}$. To prove
$\Pi^1_1\text{-}\CA$, it is sufficient to
form the set of $i\in\setN$ such that $T_i$ has an infinite path~\cite[Lemma
VI.1.1]{Simpson-2009}.  Let $A$ be the set of all pairs $\langle i,
\sigma \rangle$ such that $\sigma \in T_i$, along with one
distinguished element $z$ that is not a pair.  Let $\varphi(X)$ be the
formula $z \not \in X$, which has no parameters provided that $z$ is
coded by a standard natural number. Clearly, $\varphi$ has the finite
character property.

Write $A - \{z\} = \{a_i : i \in \N \}$, and define a nondeterministic
finitary closure operator $N = \langle N_i : i \in \N \rangle$ as
follows.  For each $j \in \N$, if $a_j = \langle i,\sigma \rangle$,
then
\begin{itemize}
\item if $\sigma$ is a dead end in $T_i$, let $N_j = \langle \{\langle
i, \sigma \rangle\}, \{z\}\rangle$;

\item if $\sigma$ is not a dead end in $T_i$, let
\[ 
N_j = \langle \{\langle i ,\sigma \rangle \}, \{\langle i, \tau
\rangle : \tau \in T_i \land \tau \succ \sigma \land |\tau| =
|\sigma| + 1\}\rangle.
\]

\end{itemize} 
Notice that $N$ can be formed by arithmetical
comprehension.

Suppose $B$ is an $N$-closed subset of $A$ that satisfies $\varphi$
(i.e., does not contain $z$). Then, for any $i$, whenever $\langle i,
\sigma \rangle$ is in $B$ there is some immediate extension $\tau$ of
$\sigma$ in $T_i$ such that $\langle i, \tau \rangle$ is in~$B$. Thus
if $\langle i, \sigma \rangle$ is in $B$ then there is an infinite
path through $T_i$ extending~$\sigma$.  So in particular, if $\langle
i, \emptyset\rangle$ is in $B$ then $T_i$ has an infinite path.
Conversely, if $f$ is an infinite path through $T_i$, then $B \cup \{
\langle i, f \res n \rangle : n \in \N \}$ is $N$-closed and satisfies
$\varphi$.

Because $\emp$ is $N$-closed and satisfies $\varphi$, we may apply
$\mathsf{NCE}$ for quantifier-free formulas to get a maximal extension
of it within~$A$.  By the previous paragraph and the maximality of
$B$, $T_i$ has a path if and only if $\langle i, \emptyset\rangle \in
B$. Thus, the set of $i$ such that $T_i$ has a path exists, as
desired.
\end{proof}

Our final results characterize the strength of
$\mathsf{NCE}$ for formulas higher in the analytical
hierarchy. 

\begin{prop}\label{P:ncehigher}
For each $n \geq 1$,
\begin{enumerate}
\item $\Sigma^1_n\text{-}\mathsf{NCE}$ and $\Pi^1_n\text{-}\mathsf{NCE}$
are provable in  $\Pi^1_n\text{-}\mathsf{CA}_0$;
\item $\Delta^1_n\text{-}\mathsf{NCE}$ is provable in
$\Delta^1_n\text{-}\mathsf{CA}_0$.
\end{enumerate}
\end{prop}
\begin{proof}

We prove part (1), the proof of part (2) being similar.  Let $\varphi(X)$
be a $\Sigma^1_n$ formula of finite character, 
respectively a $\Pi^1_n$ such formula.  Let
$N$ be a nondeterministic closure operator,  let $A$
be any set, and let $C$ be an $N$-closed subset of $A$ such
that $\varphi(C)$ holds.

By Lemma~4.5, let $\widehat{\varphi}$ be a $\Sigma^1_n$
formula, respectively a $\Pi^1_n$ formula, such that
\[
(\forall X)(\forall n)[X = D_n \to (\varphi(X)
\leftrightarrow \widehat{\varphi}(n))].
\]
We may use $\Pi^1_n$ comprehension to
form the set $W = \{ n : \widehat{\varphi}(n)\}$. Define 
$\psi(X)$ to be the arithmetical formula $(\forall n)[D_n \subseteq X \to n \in
W]$. 

We claim that for  every set $X$, $\psi(X)$ holds if and
  only if $\varphi(X)$ holds.  The definitions of $W$ and
$\psi$
ensure that $\psi(X)$ holds if and only if $\varphi(D_n)$ holds for
every finite $D_n \subseteq X$, which is true if and only if
$\varphi(X)$ holds because $\varphi$ has finite character.  This establishes
the claim.

By the claim, $\psi$ is a property of finite character and
$\psi(C)$ holds. Using $\Sigma^1_1\text{-}\mathsf{NCE}$, which is provable in
$\Pi^1_1\text{-}\mathsf{CA}_0$ by Proposition \ref{p:nceprovable} and thus  in
$\Pi^1_n\text{-}\mathsf{CA}_0$, there is a maximal
$N$-closed subset $B$ of $A$ extending~$C$ with
property~$\psi$. Again by the claim, $B$ is a maximal
$N$-closed subset of $A$ extending $B$ with
property~$\varphi$.
\end{proof}

\begin{cor}
The following are provable in $\RCAo$:
\begin{enumerate}
\item for each $n \geq 1$, $\Delta^1_n\text{-}\CA$ is equivalent to
$\Delta^1_n\text{-}\mathsf{NCE}$;
\item for each $n \geq 1$, $\Pi^1_n\text{-}\CA$ is equivalent to
$\Pi^1_n\text{-}\mathsf{NCE}$ and to $\Sigma^1_n\text{-}\mathsf{NCE}$;
\item $\mathsf{Z}_2$ is equivelent to $\NCE$.
\end{enumerate}
\end{cor}
\begin{proof}
The implications from $\Delta^1_n\text{-}\CA$, $\Pi^1_n\text{-}\CA$, and
$\mathsf{Z}_2$ follow by Proposition~\ref{P:ncehigher}. On the other hand, 
each restriction of $\NCE$ trivially implies the corresponding restriction of $\FCP$,
so the reversals follow by Corollary~\ref{c:fcpstrength}.
\end{proof}

\begin{rem}
  The characterizations in this section shed light on the role of the
  closure operator in the principles $\AL$ and $\NCE$. For
  $n \geq 1$, we have shown that $\Sigma^1_n\text{-}\FCP$,
  $\Sigma^1_n\text{-}\AL$, and $\Sigma^1_n\text{-}\NCE$ are
  all equivalent over $\RCAo$. However, $\QF\text{-}\FCP$ is
  provable in $\RCAo$, $\QF\text{-}\AL$ is equvalent to
  $\ACAo$ over $\RCAo$, and $\QF\text{-}\NCE$ is equivalent
  to $\Pi^1_1\text{-}\CA$ over $\RCAo$. Thus the closure
  operators in the stronger principles serve as a sort of
replacement for arithmetical quantification in the case of
$\AL$, and for $\Sigma^1_1$ quantification in the case of
$\NCE$. This allows these principles to have greater strength
than might be suggested by the property of finite character
alone.   At higher levels of the analytical hierarchy, the
principles become equivalent because the 
complexity of the property of finite character overtakes the
complexity of the closure notions.
\end{rem}
\section{Questions}

In this section we summarize the principal questions left over from
our investigation.  
These concern the precise strength of $F\IP$ and
the principles
$\overline{D}_n\IP$.  While we have closely located these principles'
position in the structure of statements lying between $\RCA$ and
$\ACA$, we do not know the answers to the following questions.

\begin{quest} Does $\overline{D}_2\IP$ imply $F\IP$ over $\RCA$?
Does $\overline{D}_n\IP$ imply $\overline{D}_{n+1}\IP$?
\end{quest}

\begin{quest}\label{Q:AMT_to_FIP?}  Does $\AMT$ imply
$\overline{D}_2\IP$ over $\RCA$?  Does $\OPT$ imply $\overline{D}_2\IP$?
\end{quest}

\noindent By Proposition~\ref{P:Gen_to_FIP}, the first part of the
Question~\ref{Q:AMT_to_FIP?} has an affirmative answer over $\RCA +
\mathsf{I}\Sigma^0_2$.  For the second part, it may be easier to ask
whether the implication can at least be shown to hold in $\om$-models.
An affirmative answer would likely follow from an affirmative answer
to the following question.

\begin{quest} Given a computable nontrivial family $A$, does every set
of hyperimmune degree compute a maximal subfamily of $A$ with the $F$
intersection property (or at least with the $\overline{D}_2$
intersection property)?
\end{quest}

\noindent We conjecture the answer to be no.

Our final question is less directly related to
our investigation. We mention it in view of Proposition \ref{p:alextend}
above.

\begin{quest}
What is the strength of the principle asserting that every proper ideal on
a countable join-semilattice extends to a maximal proper ideal?
\end{quest}

\noindent  This question is further motivated by work of
Turlington~\cite[Theorem 2.4.11]{Turlington-2010}
on the similar problem of constructing prime
ideals on computble lattices.  However, because a maximal ideal on a
countable lattice need not be a prime ideal, Turlington's results do not
directly resolve our question.

\bibliographystyle{amsplain} \bibliography{Choice}

\end{document}